\documentclass[10pt,reqno,final]{article}

\usepackage{graphicx}
\usepackage{amsmath,amsfonts,amssymb,amsthm,version}
\usepackage{mathrsfs,fancybox,pifont}
\usepackage{graphicx}
\usepackage{url,hyperref}
\usepackage[notcite,notref]{showkeys}
\usepackage{color}
\usepackage{subfigure,multirow}
\usepackage{epstopdf}
\usepackage{cases}
\usepackage{mathtools}
\usepackage{algorithm,algorithmic}
\usepackage{authblk}
\usepackage{lipsum}
\usepackage{float}
\allowdisplaybreaks

\theoremstyle{plain}
\newtheorem{theorem}{Theorem}[section]
\newtheorem{lemma}{Lemma}[section]

\newtheorem{proposition}[theorem]{Proposition}

\theoremstyle{definition}
\newtheorem{definition}{Definition}[section]

\theoremstyle{remark}
\newtheorem{remark}{Remark}[section]

\newcommand{\R}{\mathbb R}
\newcommand{\la}{\mathcal{L}^{\alpha}}

\newcommand{\al}{\alpha}

\newcommand{\bfd}{\mathcal{D}_{\lambda}^{\alpha}}

\begin{document}
    \title{{Boundedness of fractional integrals and fractional derivatives on Laguerre Lipschitz spaces}}

    \author{ He Wang \ \ \   Jizheng Huang \ \ \   Yu Liu\footnote{Corresponding author}} \date{}
    \maketitle

    {\bf Abstract:}\ In this paper, we study the boundedness of a class of fractional integrals and derivatives associated with
    Laguerre polynomial expansions on Laguerre Lipschitz spaces.
    The consideration of such operators is motivated by the study of corresponding results on Gaussian Lipschitz spaces.
    The key idea used here is to develop the Poisson integral theory in the Laguerre setting.


    \section{Introduction}
    As we know, Lipschitz space is an class of  important function spaces which have been used in the fields of harmonic analysis and PDEs and
     it  has been
     intensively studied. As its generalization,
    the Gaussian Lipschitz space was defined by Gatto and Urbina \cite{Gatto} in
    terms of the Ornstein-Uhlenbeck Poisson kernel. In 2016, Liu and
    Sj\"ogren gave a characterization of this space via a combination of
    ordinary Lipschitz continuity conditions (see \cite{Liu2016}). In a
    subsequent paper \cite{Liu2017}, they similarly investigated  a Lipschitz
    space in the same setting and as applications, they also
    characterized that spaces by means of a Lipschitz-type continuity
    condition. Furthermore, the Lipschitz space associated with other
    operators has also been investigated by some scholars, see
    \cite{Kallel,zhang} and the references therein. Our investigation is
    devoted to Laguerre Lipschitz spaces and their associated issues
    that expand beyond those of Gaussian Lipschitz spaces (see
    \cite{Gatto}). In order to provide a foundation for our main
    findings, we introduce some fundamental concepts for the Laguerre
    operator (see \cite{Graczyk}).

    Consider the Euclidean space $\R_{+}^{d}=(0,\infty)^d$ endowed with the Laguerre measure $\mu_{\al}$, which is defined as
    \begin{equation*}
    \mathrm{d}\mu_{\al}(x)=\prod_{i=1}^{d}\frac{{x_i}^{\al_i}e^{-x_i}}{\Gamma(\al_i+1)}
    \mathrm{d}x
    \end{equation*} for a  multiindex $\alpha=(\alpha_1,\ldots,\alpha_d)$.

    The Laguerre differential operator is defined by
    \begin{equation*}
    \la=-\sum_{i=1}^{d}\bigg[x_i\frac{\partial^2}{\partial x_i^2}+(\al_i+1-x_i)\frac{\partial}{\partial x_i}\bigg],
    \end{equation*}
    which is positive and symmetric in $L^2(\R_{+}^d,\mathrm{d}\mu_{\al})$. Moreover, $\la$ has a closure which
     is self-adjoint in $L^2(\R_{+}^d,\mathrm{d}\mu_{\al})$
    and it also will be denoted by $\la$.

    Given $\al>-1$,  the one-dimensional Laguerre polynomials of type
    $\al$ are denoted by
    \begin{equation*}
    L_k^{\al}(x)=\frac{1}{k!}e^{x}x^{-\al}\frac{\mathrm{d}^k}{\mathrm{d}x^k}(e^{-x}x^{k+\al}),\ k\in\mathbb{N}, \ x>0.
    \end{equation*}
    Note that each $L_k^{\al}$ is a polynomial of degree $k$. Given a multiindex
    $\al=(\al_1,\ldots,\al_d)$ with $\al\in (-1,\infty)^d$, the $d$-dimensional Laguerre
    polynomials of type $\al$ are tensor products of the one-dimensional Laguerre
    polynomials, that is,
    \begin{equation*}
    L_k^{\al}(x)=\prod_{i=1}^d L_{k_i}^{\al_i}(x_i),\ k\in\mathbb{N}^d,\ x\in\R_{+}^d,
    \end{equation*}
    then it is well known that the Laguerre polynomials are eigenfunctions of
    $\la$, that is,
    \begin{equation}\label{eqq1}
    \la L_k^{\al}(x)=|k|L_k^{\al}(x).
    \end{equation}

    By the orthogonality of the Laguerre polynomials with respect to $\mu_{\al}$, it is easy to see that the system $\{L_k^{\al}:k\in\mathbb{N}^d\}$ constitutes
    an orthogonal basis in $L^2(\R_{+}^d,\mathrm{d}\mu_{\al})$. Thus we have an
    orthogonal decomposition:
    \begin{equation*}
    L^{2}(\R_{+}^d,\mathrm{d}\mu_{\al})=\bigoplus_{n=0}^{\infty}\mathcal{H}_n,
    \end{equation*}
    where $\mathcal{H}_n=\mathrm{lin}\{L_k^{\al}:|k|=n\}$.

    The heat semigroup generated by $\la$ and defined in $L^2(\R_{+}^d,\mathrm{d}\mu_{\al})$ is called the Laguerre semigroup
    \begin{equation*}
    T_t^{\al}=e^{-t\la},\ t\ge 0,
    \end{equation*}
   which can be defined as
    \begin{equation}\label{eq1}
    T_t^{\al}f(x)=\int_{\R_{+}^d}G_t^{\al}(x,y)f(y)\mathrm{d}\mu_{\al}(y),\ f\in
    L^p(\R_{+}^d,\mathrm{d}\mu_{\al}).
    \end{equation}
Via the Hille-Hardy formula \cite[Chapter 1, Section 1.11]{HMHL}, we
know that the kernel $G_t^{\al}(x,y)$ may be computed explicitly,
that is,
    \begin{equation}\label{eq2}
    G_t^{\al}(x,y)=\prod_{j=1}^d\frac{\Gamma(\al_j+1)}{1-e^{-t}}e^{-\frac{e^{-t}}{1-e^{-t}}(x_j+y_j)}\sqrt{e^{-t}x_j y_j}^{-\al_j}I_{\al_j}\bigg(\frac{2\sqrt{e^{-t}x_jy_j}}{1-e^{-t}}\bigg),
    \end{equation}
    where $I_{\nu}$ denotes the modified Bessel function of the first kind and order $\nu$ (see \cite{NNL}).

    The paper is organized as follows. In Section \ref{section2}, we introduce
    the Poisson-Laguerre semigroup, which is the semigroup subordinated to the
    Laguerre semigroup. By obtaining  some properties of the modified Bessel function
    (\ref{eq3}) and (\ref{eq4}), we deduce the representation of $T_t^{\al}f(x)$:
    \begin{equation*}
    T_{t}^{\al}f(x)\approx\frac{2^{|\al|}}{(1-e^{-t})^{d+|\al|}}
    \int_{\R_{+}^d}\prod_{j=1}^d{y_j}^{\al_j}e^{-\frac{e^{-t}x_j+y_j}{1-e^{-t}}}f(y)\mathrm{d}y,\quad \mathrm{for}\ \frac{\sqrt{e^{-t}x_jy_j}}{1-e^{-t}}\le 1,\ j=1,2,\ldots,d,
    \end{equation*}
    and
    \begin{equation*}
    \begin{split}
    T_{t}^{\al}f(x)&=\frac{e^{(|\al|+\frac{d}{2})\frac{t}{2}}}{(2\sqrt{\pi})^d(1-e^{-t})^{\frac{d}{2}}}\int_{\R_{+}^d}\prod_{j=1}^{d}\frac{y_j^{\frac{\al_j}{2}-\frac{1}{4}}}{x_j^{\frac{\al_j}{2}+\frac{1}{4}}}e^{-\frac{(\sqrt{e^{-t}x_j}-\sqrt{y_j})^2}{1-e^{-t}}}
    \bigg[1-\frac{[\al_j,1]}{2}\bigg(\frac{2\sqrt{e^{-t}xy}}{1-e^{-t}}\bigg)^{-1}+\frac{[\al_j,2]}{4}\\
    &\quad\times\bigg(\frac{2\sqrt{e^{-t}xy}}{1-e^{-t}}\bigg)^{-2}
    +O\bigg(\bigg(\frac{2\sqrt{e^{-t}xy}}{1-e^{-t}}\bigg)^{-3}\bigg)\bigg]
    f(y)\mathrm{d}y,\quad\mathrm{for}\ \frac{\sqrt{e^{-t}x_jy_j}}{1-e^{-t}}>1,\ j=1,2,\ldots,d,
    \end{split}
    \end{equation*}
and $\mathrm{for}\ \frac{\sqrt{e^{-t}x_jy_j}}{1-e^{-t}}\le 1\ (j=1,2,\ldots,k), \frac{\sqrt{e^{-t}x_jy_j}}{1-e^{-t}}>1\ (j=k+1,k+2,\ldots,d), \mathrm{where}\ k\in\mathbb{N},0\le k\le d$,
\begin{equation*}
    \begin{split}
        T_{t}^{\al}f(x)&=\frac{2^{\al_1+\al_2+\cdots+\al_k}e^{(\al_{k+1}+\al_{k+2}+\cdots+\al_d+\frac{d-k}{2})\frac{t}{2}}}{(2\sqrt{\pi})^{d-k}(1-e^{-t})^{\al_1+\al_2+\cdots+\al_k+\frac{k+d}{2}}}\int_{\R_{+}^d}\prod_{j=1}^k{y_j}^{\al_j}e^{-\frac{e^{-t}x_j+y_j}{1-e^{-t}}}
        \prod_{j=k+1}^{d}\frac{y_j^{\frac{\al_j}{2}-\frac{1}{4}}}{x_j^{\frac{\al_j}{2}+\frac{1}{4}}}e^{-\frac{(\sqrt{e^{-t}x_j}-\sqrt{y_j})^2}{1-e^{-t}}}\\
        &\quad\times\bigg[1-\frac{[\al_j,1]}{2}\bigg(\frac{2\sqrt{e^{-t}xy}}{1-e^{-t}}\bigg)^{-1}+\frac{[\al_j,2]}{4}\bigg(\frac{2\sqrt{e^{-t}xy}}{1-e^{-t}}\bigg)^{-2}
        +O\bigg(\bigg(\frac{2\sqrt{e^{-t}xy}}{1-e^{-t}}\bigg)^{-3}\bigg)\bigg]f(y)\mathrm{d}y.
    \end{split}
\end{equation*}

    Using the relation between the Poisson-Laguerre semigroup and the
    Laguerre semigroup in (\ref{eq5.1}), we obtain the kernel estimation of the Poisson-Laguerre semigroup. In Lemma \ref{lem2.1},
     we derive the $L^1(\R_{+}^d)$ estimate of the derivative of the Poisson-Laguerre kernel $p_t^{\al}(x,y)$ with
     respect to the time variable $t$ according to two distinct cases.

    In Section \ref{sec-2}, we obtain the equivalence between   conditions (\ref{eq-14}) and (\ref{eq-15}), see Proposition \ref{pro3.1}.
     In Definition \ref{def3.1}, we introduce Laguerre Lipschitz spaces denoted by $Lip_{\beta}(\mu_{\al})$ and prove
      two basic properties: the monotonicity property (see Proposition \ref{pro3.2}) and the approximation property (see Proposition \ref{pro3.3}).

    Section \ref{sec-3} is devoted to the boundedness of fractional integrals
    and fractional derivatives associated with Laguerre expansions on $Lip_{\beta}(\mu_{\al})$. In Theorem \ref{thm3.1}, we obtain the boundedness of the Bessel Laguerre potential $\mathcal{J}_{\lambda}^{\al}$ from $Lip_{\beta}(\mu_{\al})$ to $Lip_{\beta+\lambda}(\mu_{\al})$ for $\beta,\lambda>0$. For $0<\lambda<\beta<1$, Theorem \ref{thm4.2} deduces the boundedness of the Laguerre fractional derivative $D_{\lambda}^{\al}$ from $Lip_{\beta}(\mu_{\al})$ to $Lip_{\beta-\lambda}(\mu_{\al})$. Similarly, the boundedness of the Bessel Laguerre fractional derivative is proved in Theorem \ref{thm3.3}. Finally, Theorem \ref{thm4.4} reveals the boundedness of above two operators on Lipschitz spaces for $1\le \lambda<\beta$.

    Throughout this paper, we will use $c$ and $C$ to denote the positive constants, which are independent of main parameters and may be different at each occurrence. By $X\approx Y$, we mean that $Y\lesssim X\lesssim Y$, where the second estimate means that there exists a positive constant $C$, independent of main parameters, such that $X\le CY$. Similarly, one writes $V\gtrsim U$ for $V\ge cU$. We also use the following notation in our paper: $\mathrm{N}=\{0,1,2,\ldots\}$.
    \section{Some results for Laguerre semigroups}\label{section2}
    Firstly, we start by recalling some basic notations and facts from the Laguerre semigroup theory. For more details, see \cite{BCT,Gszego} and the references therein.

    We are going to quote two properties of the modified Bessel function $I_{\nu}$ that we will need in the subsequent proofs (see \cite[Chap. 5]{NNL}),
     \begin{equation}\label{eq3}
        I_{\nu}(z)\approx z^{\nu}\ \mathrm{as}\ z\to 0^{+},
     \end{equation}
     for $n=0,1,2,\ldots$, and
     \begin{equation}\label{eq4}
        \sqrt{z}I_{\nu}(z)=\frac{e^z}{\sqrt{2\pi}}\bigg(\sum_{r=0}^{n}(-1)^r
        [\nu,r](2z)^{-r}+O(z^{-n-1})\bigg),
     \end{equation}
     where $[\nu,0]=1$ and $[\nu,r]=\frac{(4\nu^2-1)(4\nu^2-3^2)\cdots(4\nu^2-(2r-1)^2)}{2^{2r}\Gamma(r+1)}$, $r=1,2,\ldots$.

     It follows from \cite{Graczyk} that $\{T_t^{\al}\}_{t\ge 0}$ is a symmetric diffusion semigroup. In particular $T_t^{\al}\mathbf{1}=\mathbf{1}$ and $T_0^{\al}f(x)=f(x)$\ a.e. $x\in\R_{+}^d$.
     The Poisson-Laguerre semigroup $P_t^{\al}=e^{-t\sqrt{\la}}$,\ $t\ge 0$, can be defined from $\{T_t^{\al}\}_{t\ge 0}$ by subordination as
     \begin{equation}\label{eq5.1}
        P_t^{\al}f(x)=\frac{1}{\sqrt{\pi}}\int_{0}^{\infty}\frac{e^{-u}}{\sqrt{u}}T_{\frac{t^2}{4u}}^{\al}f(x)\mathrm{d}u.
     \end{equation}
     Furthermore, by a change of   variables, we obtain the subsequent representation
     \begin{equation}\label{eq7.1}
        P_t^{\al}f(x)=\int_0^{\infty}T_s^{\al}f(x)\mu_t^{(\frac{1}{2})}\mathrm{d}s,
     \end{equation}
     where the one-side stable measure on $(0,\infty)$ of order $\frac{1}{2}$ is defined by
     \begin{equation*}
        \mu_t^{(\frac{1}{2})}\mathrm{d}s:=\frac{t}{2\sqrt{\pi}}\frac{e^{-\frac{t^2}{4s}}}{s^{\frac{3}{2}}}\mathrm{d}s=g(t,s)\mathrm{d}s.
     \end{equation*}

     Using (\ref{eq1}) and the expression of the Laguerre heat kernel (\ref{eq2}),
     we obtain
     \begin{equation}\label{eq5}
      \begin{split}
      T_t^{\al}f(x)&=\int_{\R_{+}^d}\prod_{j=1}^d\frac{\Gamma(\al_j+1)}{1-e^{-t}}e^{-\frac{e^{-t}}{1-e^{-t}}(x_j+y_j)}\sqrt{e^{-t}x_j y_j}^{-\al_j}I_{\al_j}\bigg(\frac{2\sqrt{e^{-t}x_jy_j}}{1-e^{-t}}\bigg)f(y)\mathrm{d}\mu_{\al}(y)\\
      &=\frac{1}{(1-e^{-t})^d}\int_{\R_{+}^d}\prod_{j=1}^d
      e^{-\frac{e^{-t}}{1-e^{-t}}(x_j+y_j)-y_j}\sqrt{e^{-t}x_j y_j}^{-\al_j}{y_{j}}^{\al_j}I_{\al_j}\bigg(\frac{2\sqrt{e^{-t}x_jy_j}}{1-e^{-t}}\bigg)f(y)
      \mathrm{d}y\\
      &=\frac{1}{(1-e^{-t})^d}\int_{\R_{+}^d}\prod_{j=1}^d\bigg(\frac{y_j}{x_j}\bigg)^{\frac{\al_j}{2}}
      e^{-\frac{e^{-t}x_j+y_j}{1-e^{-t}}+\frac{\al_jt}{2}}I_{\al_j}\bigg(\frac{2\sqrt{e^{-t}x_jy_j}}{1-e^{-t}}\bigg)f(y)\mathrm{d}y.
      \end{split}
     \end{equation}

     From (\ref{eq5.1}), (\ref{eq5}) and after the change of the variable
     $r=e^{-\frac{t^2}{4u}}$ and exchanging the order of integration, we obtain
     \begin{equation*}
        \begin{split}
P_t^{\al}f(x)&=\frac{1}{\sqrt{\pi}}\int_{0}^{\infty}\frac{e^{-u}}{\sqrt{u}(1-e^{-\frac{t^2}{4u}})^d}\int_{\R_+^d}\prod_{j=1}^{d}\bigg(\frac{y_j}{x_j}\bigg)^{\frac{\al_j}{2}}e^{-\frac{e^{-\frac{t^2}{4u}}x_j+y_j}{1-e^{-\frac{t^2}{4u}}}+\frac{\al_jt^2}{8u}}I_{\al_j}\bigg(\frac{2\sqrt{e^{-\frac{t^2}{4u}}x_jy_j}}{1-e^{-\frac{t^2}{4u}}}\bigg)f(y)\mathrm{d}y\mathrm{d}u\\
&=\frac{1}{2\sqrt{\pi}}\int_{\R_{+}^d}\int_0^1\frac{te^{\frac{t^2}{4\log r}}}{(1-r)^d(-\log r)^{\frac{3}{2}}}\prod_{j=1}^{d}\bigg(\frac{y_j}{x_j}\bigg)^{\frac{\al_j}{2}}e^{-\frac{rx_j+y_j}{1-r}-\frac{\al_j\log r}{2}}I_{\al_j}\bigg(\frac{2\sqrt{rx_jy_j}}
{1-r}\bigg)\frac{\mathrm{d}r}{r}f(y)\mathrm{d}y\\
&=\int_{\R_{+}^d}p_t^{\al}(x,y)f(y)\mathrm{d}y,
        \end{split}
     \end{equation*}
where
\begin{equation}\label{eq9-1}
     p_t^{\al}(x,y):=\frac{1}{2\sqrt{\pi}}\int_0^1\frac{te^{\frac{t^2}{4\log r}}}{(1-r)^d(-\log r)^{\frac{3}{2}}}\prod_{j=1}^{d}\bigg(\frac{y_j}{x_j}\bigg)^{\frac{\al_j}{2}}e^{-\frac{rx_j+y_j}{1-r}-\frac{\al_j\log r}{2}}I_{\al_j}\bigg(\frac{2\sqrt{rx_jy_j}}
     {1-r}\bigg)\frac{\mathrm{d}r}{r}.
\end{equation}

    The family $\{P_t^{\al}\}_{t\ge 0}$ is also a symmetric diffusion semigroup equipped with the infinitesimal generator $-(-\la)^{\frac{1}{2}}$. In particular,
    $P_t^{\al}\mathbf{1}=\mathbf{1}$ due to (\ref{eq7.1}) and $T_t^{\al}\mathbf{1}=\mathbf{1}$.

    By (\ref{eqq1}), we obtain
    \begin{equation*}
        T_t^{\al}L_k^{\al}=e^{-t|k|}L_k^{\al}
    \end{equation*}
    and
    \begin{equation*}
        P_{t}^{\al}L_k^{\al}=e^{-t\sqrt{|k|}}L_k^{\al},
    \end{equation*}
    which imply that $L_k^{\al}$ are eigenfunctions of $\{T_t^{\al}\}_{t\ge 0}$ and $\{P_t^{\al}\}_{t\ge 0}$, respectively.
    Since the maximal operator $(T^{\al})^*f(x)=\sup_{t>0}|T_t^{\al}f(x)|$ satisfies the weak type $(1,1)$ inequality with respect to the measure
    $\mathrm{d}\mu_{\al}$ (see \cite{Dinger}), therefore,
    \begin{equation*}
        T_0^{\al}f(x)=\lim_{t\rightarrow 0^{+}}T_t^{\al}f(x)=f(x)\ \mathrm{a.e.}\
        x\in\R_{+}^d
    \end{equation*}
    and
    \begin{equation*}
        T_{\infty}^{\al}f(x)=\lim_{t\rightarrow \infty}T_t^{\al}f(x)=\int_{\R_{+}^d}
        f(y)\mathrm{d}\mu_{\al}(y)\ \mathrm{a.e.}\ x\in\R_{+}^d
    \end{equation*}
    hold for all $f\in L^1(\R_{+}^d,\mathrm{d}\mu_{\al})$, and hence hold for all
    $f\in L^p(\R_{+}^d,\mathrm{d}\mu_{\al})$ with $1\le p\le\infty$ due
    to $L^q(\R_{+}^d,\mathrm{d}\mu_{\al})\subset L^p(\R_{+}^d,\mathrm{d}\mu_{\al})$ for
    $p\le q$. Please refer to \cite{STH} for more results about Laguerre expansions.

    Secondly, we will need the following lemma for the $L^1$-norm
    of the derivatives of the kernel $p_t^{\al}(x,y)$.
    \begin{lemma}\label{lem2.1}
        Assume that $p_t^{\al}(x,y)$ is the Poisson-Laguerre kernel.
         \item{\rm (i)} If $\frac{\sqrt{e^{-t}x_jy_j}}{1-e^{-t}}\le 1$, $j=1,2,\ldots,d$, then
        \begin{equation}\label{eq10}
            \int_{\R_{+}^d}\bigg|\frac{\partial p_t^{\al}(x,y)}{\partial t}\bigg|\mathrm{d}y\le\frac{C}{t},
        \end{equation}
        where $C$ is a constant independent of $x$ and $t$.
        \item{\rm (ii)} If $\frac{\sqrt{e^{-t}x_jy_j}}{1-e^{-t}}> 1$, $j=1,2,\ldots,d$, and $\al\in(-\frac{1}{2},\infty)^d$, then
        \begin{equation*}
        \int_{\R_{+}^d}\bigg|\frac{\partial p_t^{\al}(x,y)}{\partial t}\bigg|\mathrm{d}y\le\frac{C}{t},
        \end{equation*}
        where $C$ is a constant independent of $x$ and $t$.
        \item{\rm (iii)} If $\frac{\sqrt{e^{-t}x_jy_j}}{1-e^{-t}}\le 1$, $j=1,2,\ldots,k$,
        and $\frac{\sqrt{e^{-t}x_jy_j}}{1-e^{-t}}> 1$ with $\al_j>-\frac{1}{2}$ for $j=k+1,k+2,\ldots,d$, where $k\in\mathbb{N}$, $0\le k\le d$, then
         \begin{equation*}
            \int_{\R_{+}^d}\bigg|\frac{\partial p_t^{\al}(x,y)}{\partial t}\bigg|\mathrm{d}y\le\frac{C}{t}.
         \end{equation*}

        Additionally, for any positive integer $k$ and $\al\in (-\frac{1}{2},\infty)^d$, we obtain
        \begin{equation}\label{eq11}
        \int_{\R_{+}^d}\bigg|\frac{\partial^k p_t^{\al}(x,y)}{\partial t^k}\bigg|\mathrm{d}y\le\frac{C}{t^k}.
        \end{equation}
    \end{lemma}
    \begin{proof}
(i) We first prove (\ref{eq10}). By the expression of the Poisson-Laguerre kernel (\ref{eq9-1}) we have
        \begin{equation}\label{eq-12}
\frac{\partial p_t^{\al}(x,y)}{\partial t}=\frac{1}{2\sqrt{\pi}}\int_0^1\frac{(1+\frac{t^2}{2\log r})e^{\frac{t^2}{4\log r}}}{(1-r)^d(-\log r)^{\frac{3}{2}}}\prod_{j=1}^{d}
\bigg(\frac{y_j}{x_j}\bigg)^{\frac{\al_j}{2}}e^{-\frac{rx_j+y_j}{1-r}-\frac{\al_j\log r}{2}}I_{\al_j}
\bigg(\frac{2\sqrt{rx_jy_j}}{1-r}\bigg)\frac{\mathrm{d}r}{r}.
        \end{equation}
From (\ref{eq3}) and by exchanging the order of integration, in case
$\frac{\sqrt{e^{-t}x_jy_j}}{1-e^{-t}}\le 1$, $j=1,2,\ldots,d$, we
obtain
\begin{align*}
    &\int_{\R_{+}^d}\bigg|\frac{\partial p_t^{\al}(x,y)}{\partial t}\bigg|\mathrm{d}y\\
    &\le\frac{1}{2\sqrt{\pi}}\int_{\R_{+}^d}\int_0^1\frac{|1+\frac{t^2}{2\log r}|e^{\frac{t^2}{4\log r}}}{(1-r)^d(-\log r)^{\frac{3}{2}}}\prod_{j=1}^{d}\bigg(\frac{y_j}{x_j}\bigg)^{\frac{\al_j}{2}}
    e^{-\frac{rx_j+y_j}{1-r}-\frac{\al_j\log r}{2}}\bigg|I_{\al_j}
    \bigg(\frac{2\sqrt{rx_jy_j}}{1-r}\bigg)\bigg|\frac{\mathrm{d}r}{r}\mathrm{d}y\\
    &\approx\frac{1}{2\sqrt{\pi}}\int_{\R_{+}^d}\int_0^1\frac{|1+\frac{t^2}{2\log r}|e^{\frac{t^2}{4\log r}}}{(1-r)^d(-\log r)^{\frac{3}{2}}}\prod_{j=1}^{d}\bigg(\frac{y_j}{x_j}
    \bigg)^{\frac{\al_j}{2}}e^{-\frac{rx_j+y_j}{1-r}-\frac{\al_j\log r}{2}}
    \bigg(\frac{2\sqrt{rx_jy_j}}{1-r}\bigg)^{\al_{j}}\frac{\mathrm{d}r}{r}\mathrm{d}y\\
    &=\frac{2^{|\al|-1}}{\sqrt{\pi}}\int_0^1\frac{e^{\frac{t^2}{4\log r}}}{(-\log r)^{\frac{3}{2}}}\bigg|1+\frac{t^2}{2\log r}\bigg|\int_{\R_{+}^d}\prod_{j=1}^d
    \frac{e^{-\frac{rx_j+y_j}{1-r}}}{(1-r)^{\al_j+1}}y_j^{\al_j}\mathrm{d}y\frac{\mathrm{d}r}{r}\\
    &\le\frac{2^{|\al|-1}\prod_{j=1}^{d}\Gamma(\al_j+1)}{\sqrt{\pi}}
    \int_0^1\frac{e^{\frac{t^2}{4\log r}}}{(-\log r)^{\frac{3}{2}}}\bigg|1+\frac{t^2}{2\log r}\bigg|e^{-\frac{r|x|}{1-r}}\frac{\mathrm{d}r}{r}\\
    &\le C_{\al}\int_0^1\frac{e^{\frac{t^2}{4\log r}}}{(-\log r)^{\frac{3}{2}}}\bigg|1+\frac{t^2}{2\log r}\bigg|\frac{\mathrm{d}r}{r},
\end{align*}
where in the second inequality we have used the fact that
        \begin{equation}\label{eq13-1}
        \int_{\R_{+}^d}\prod_{j=1}^d
        \frac{e^{-\frac{rx_j+y_j}{1-r}}}{(1-r)^{\al_j+1}}y_j^{\al_j}   \mathrm{d}y=\prod_{j=1}^d\Gamma(\al_j+1)e^{-\frac{rx_j}{1-r}}.
        \end{equation}
        Therefore, to show (\ref{eq10}) it suffices to prove
        \begin{equation}\label{eq12}
            \int_0^1\frac{e^{\frac{t^2}{4\log r}}}{(-\log r)^{\frac{3}{2}}}\bigg|1+\frac{t^2}{2\log r}\bigg|
            \frac{\mathrm{d}r}{r}\le \frac{C}{t}.
        \end{equation}
        Taking $s=-\log r$, we get
        \begin{align*}
            \int_{0}^1\frac{e^{\frac{t^2}{4\log r}}}{(-\log r)^{\frac{3}{2}}}\bigg|1+\frac{t^2}{2\log r}\bigg|\frac{\mathrm{d}r}{r}&=\int_{0}^{\infty}\frac{e^{-\frac{t^2}{4s}}}{s^{\frac{3}{2}}}\bigg|1-\frac{t^2}{2s}\bigg|\mathrm{d}s\\
            &\le \int_{0}^{\infty}\frac{e^{-\frac{t^2}{4s}}}{s^{\frac{3}{2}}}
            \mathrm{d}s+\int_{0}^{\infty}\frac{e^{-\frac{t^2}{4s}}}{s^{\frac{3}{2}}}
            \frac{t^2}{2s}\mathrm{d}s\\
            &=:I+II.
        \end{align*}
        For $I$, via the change of the variable $v=\frac{t^2}{4s}$, we have
        \begin{align*}
            I&=\int_{0}^{\infty}e^{-v}\bigg(\frac{t^2}{4v}\bigg)^{-\frac{3}{2}}
            \frac{t^2}{4v^2}\mathrm{d}v
            =\int_{0}^{\infty}e^{-v}\frac{(4v)^{\frac{3}{2}}}{t^3}\frac{t^2}{4v^2}\mathrm{d}v\\
            &=\frac{C}{t}\int_{0}^{\infty}e^{-v}v^{-\frac{1}{2}}\mathrm{d}v
            =\frac{C\Gamma(\frac{1}{2})}{t}=\frac{C^{\prime}}{t}.
        \end{align*}
        Similarly,
        \begin{align*}
            II
            &=2\int_{0}^{\infty}e^{-v}\bigg(\frac{t^2}{4v}\bigg)^{-\frac{3}{2}}v\frac{t^2}{4v^2}\mathrm{d}v
            =2\int_{0}^{\infty}e^{-v}\frac{(4v)^{\frac{3}{2}}}{t^3}v\frac{t^2}{4v^2}\mathrm{d}v\\
            &=\frac{C}{t}\int_{0}^{\infty}e^{-v}v^{\frac{1}{2}}\mathrm{d}v
            =\frac{C\Gamma(\frac{3}{2})}{t}=\frac{C^{\prime}}{t}.
        \end{align*}
     (ii) Similarly, using (\ref{eq4}) and (\ref{eq-12}), we obtain the following estimate
    \begin{equation*}
        \begin{split}
            &\frac{\partial p_t^{\al}(x,y)}{\partial t}\\
            &=\frac{1}{2\sqrt{\pi}}\int_0^1\frac{(1+\frac{t^2}{2\log r})e^{\frac{t^2}{4\log r}}}{(1-r)^d(-\log r)^{\frac{3}{2}}}\prod_{j=1}^{d}
            \bigg(\frac{y_j}{x_j}\bigg)^{\frac{\al_j}{2}}e^{-\frac{rx_j+y_j}{1-r}-\frac{\al_j\log r}{2}}I_{\al_j}\bigg(\frac{2\sqrt{rx_jy_j}}{1-r}\bigg)\frac{\mathrm{d}r}{r}\\
            &=\frac{1}{2\sqrt{\pi}}\int_0^1\frac{(1+\frac{t^2}{2\log r})e^{\frac{t^2}{4\log r}}}{(1-r)^d(-\log r)^{\frac{3}{2}}}\prod_{j=1}^{d}
            \bigg(\frac{y_j}{x_j}\bigg)^{\frac{\al_j}{2}}e^{-\frac{rx_j+y_j}{1-r}-\frac{\al_j\log r}{2}}\bigg(\frac{2\sqrt{rx_jy_j}}{1-r}\bigg)^{-\frac{1}{2}}
            \bigg(\frac{2\sqrt{rx_jy_j}}{1-r}\bigg)^{\frac{1}{2}}\\
            &\quad\times I_{\al_j}
            \bigg(\frac{2\sqrt{rx_jy_j}}{1-r}\bigg)\frac{\mathrm{d}r}{r}\\
            &=\frac{1}{(2\sqrt{\pi})^{d+1}}\int_{0}^{1}\frac{(1+\frac{t^2}{2\log r})e^{\frac{t^2}{4\log r}}}{(1-r)^{\frac{d}{2}}(-\log r)^{\frac{3}{2}}}\prod_{j=1}^{d}
            \frac{y_j^{\frac{\al_j}{2}-\frac{1}{4}}}{x_j^{\frac{\al_j}{2}+\frac{1}{4}}}e^{-\frac{(\sqrt{rx_j}-\sqrt{y_j})^2}{1-r}}r^{-\frac{1}{2}(\al_j+\frac{1}{2})}
            \bigg[1-\frac{[\al_j,1]}{2}\bigg(\frac{2\sqrt{rx_jy_j}}{1-r}\bigg)^{-1}\\
            &\quad+\frac{[\al_j,2]}{4}\bigg(\frac{2\sqrt{rx_jy_j}}{1-r}\bigg)^{-2}+O\bigg(\bigg(\frac{2\sqrt{rx_jy_j}}{1-r}\bigg)^{-3}\bigg)\bigg]\frac{\mathrm{d}r}{r}
        \end{split}
    \end{equation*}for every $x,y\in\R_{+}^d$, $t\in(0,\infty)$ such that
    $\frac{\sqrt{e^{-t}x_jy_j}}{1-e^{-t}}>1$.

        Since $\frac{\sqrt{e^{-t}x_jy_j}}{1-e^{-t}}> 1$ for $j=1,2,\ldots,d$, or equivalently
        $\frac{\sqrt{rx_jy_j}}{1-r}> 1$ for $j=1,2,\ldots,d$, then there exists a constant $C$ such that
        \begin{equation}\label{eq-14.1}
        \bigg|1-\frac{[\al_j,1]}{2}\bigg(\frac{2\sqrt{rx_jy_j}}{1-r}\bigg)^{-1}+\frac{[\al_j,2]}{4}\bigg(\frac{2\sqrt{rx_jy_j}}{1-r}\bigg)^{-2}+O\bigg(\bigg(\frac{2\sqrt{rx_jy_j}}{1-r}\bigg)^{-3}\bigg) \bigg|\le C.
        \end{equation}
        Thus we have
        \begin{equation}\label{eq16.1}
            \bigg|\frac{\partial p_t^{\al}(x,y)}{\partial t}\bigg|\le C\int_0^1
            \frac{e^{\frac{t^2}{4\log r}}}{(-\log r)^{\frac{3}{2}}}\bigg|1+\frac{t^2}{2\log r}\bigg|\prod_{j=1}^d
            \frac{y_j^{\frac{\al_j}{2}-\frac{1}{4}}}{x_j^{\frac{\al_j}{2}+\frac{1}{4}}}\frac{e^{-\frac{(\sqrt{rx_j}-\sqrt{y_j})^2}{1-r}}          }{r^{\frac{\al_j}{2}+\frac{1}{4}}\sqrt{1-r}}\frac{\mathrm{d}r}{r}.
        \end{equation}
        On the other hand, we split the following integral into two integrals over $(0,rx_j)$ and $[rx_j,\infty)$, which are called $I_j$ and $II_j$,
         respectively. Then we have
        \begin{equation}\label{eq17.2}
            \begin{split}
            &\int_{\R_{+}^d}
            \prod_{j=1}^d\frac{y_j^{\frac{\al_j}{2}-\frac{1}{4}}}{x_j^{\frac{\al_j}{2}+\frac{1}{4}}}e^{-\frac{(\sqrt{rx_j}-\sqrt{y_j})^2}{1-r}}\mathrm{d}y\\
            &\quad=\prod_{j=1}^d\int_{0}^{\infty}
            \frac{y_j^{\frac{\al_j}{2}-\frac{1}{4}}}{x_j^{\frac{\al_j}{2}+\frac{1}{4}}}e^{-\frac{(\sqrt{rx_j}-\sqrt{y_j})^2}{1-r}}\mathrm{d}y_j\\
            &\quad=\prod_{j=1}^d\bigg(\int_{0}^{rx_j}
            \frac{y_j^{\frac{\al_j}{2}-\frac{1}{4}}}{x_j^{\frac{\al_j}{2}+\frac{1}{4}}}e^{-\frac{(\sqrt{rx_j}-\sqrt{y_j})^2}{1-r}}\mathrm{d}y_j+
            \int_{rx_j}^{\infty}\frac{y_j^{\frac{\al_j}{2}-\frac{1}{4}}}{x_j^{\frac{\al_j}{2}+\frac{1}{4}}}e^{-\frac{(\sqrt{rx_j}-\sqrt{y_j})^2}{1-r}}\mathrm{d}y_j\bigg)\\
            &\quad=:\prod_{j=1}^d(I_j+II_j).
            \end{split}
        \end{equation}
        For $I_j$, a change of the variable $\sqrt{z_j}=\frac{\sqrt{rx_j}-\sqrt{y_j}}{\sqrt{1-r}}$
        leads to
        \begin{equation}\label{eq17.1}
        \begin{split}
            I_j&=\frac{\sqrt{1-r}}{x_j^{\frac{\al_j}{2}+\frac{1}{4}}}\int_{0}^{\frac{rx_j}{1-r}}\bigg(\sqrt{rx_j}-\sqrt{(1-r)z_j}\bigg)^{\al_j+\frac{1}{2}}z_j^{-\frac{1}{2}}e^{-z_j}\mathrm{d}z_j\\
            &\le\frac{\sqrt{1-r}}{x_j^{\frac{\al_j}{2}+\frac{1}{4}}}\int_{0}^{\infty}\bigg(\sqrt{rx_j}-\sqrt{(1-r)z_j}\bigg)^{\al_j+\frac{1}{2}}z_j^{-\frac{1}{2}}e^{-z_j}
            \mathrm{d}z_j.
        \end{split}
        \end{equation}
        For $II_j$, by a change of  the variable  $\sqrt{z_j}=\frac{\sqrt{y_j}-\sqrt{rx_j}}{\sqrt{1-r}}$ again we deduce
        \begin{equation}\label{eq18-1}
            II_j=\frac{\sqrt{1-r}}{x_j^{\frac{\al_j}{2}+\frac{1}{4}}}\int_{0}^{\infty}\bigg(\sqrt{rx_j}+\sqrt{(1-r)z_j}\bigg)^{\al_j+\frac{1}{2}}z_j^{-\frac{1}{2}}e^{-z_j}
            \mathrm{d}z_j.
        \end{equation}
        By inserting (\ref{eq17.1}) and (\ref{eq18-1}) into (\ref{eq17.2}) and assuming that $\al_j>-\frac{1}{2}$ for every given $1\le j\le d$ we get
        \begin{align}\label{eq20.1}
        &\int_{\R_{+}^d}\nonumber
        \prod_{j=1}^d\frac{y_j^{\frac{\al_j}{2}-\frac{1}{4}}}{x_j^{\frac{\al_j}{2}+\frac{1}{4}}}
        e^{-\frac{(\sqrt{rx_j}-\sqrt{y_j})^2}{1-r}}\mathrm{d}y\\ \nonumber
        &\quad\le\prod_{j=1}^d\frac{\sqrt{1-r}}{x_j^{\frac{\al_j}{2}+\frac{1}{4}}}\int_0^{\infty}
        \bigg[\bigg(\sqrt{rx_j}-\sqrt{(1-r)z_j}\bigg)^{\al_j+\frac{1}{2}}+\bigg(\sqrt{rx_j}+\sqrt{(1-r)z_j}\bigg)^{\al_j+\frac{1}{2}}\bigg]z_j^{-\frac{1}{2}}e^{-z_j}
        \mathrm{d}z_j\\ \nonumber
        &\quad\le\prod_{j=1}^d\frac{\sqrt{1-r}}{x_j^{\frac{\al_j}{2}+\frac{1}{4}}}
        \int_0^{\infty}\big(2\sqrt{rx_j}\big)^{\al_j+\frac{1}{2}}z_j^{-\frac{1}{2}}e^{-z_j}
        \mathrm{d}z_j\\
        &\quad=(2\pi)^{\frac{d}{2}}2^{|\al|}(1-r)^{\frac{d}{2}}r^{\frac{|\al|}{2}+\frac{d}{4}}.
        \end{align}
        Hence, via (\ref{eq16.1}) and (\ref{eq20.1}) we have
        \begin{align*}
            &\int_{\R_{+}^d}\bigg|\frac{\partial p_t^{\al}(x,y)}{\partial t}\bigg|
            \mathrm{d}y\\
            &\quad\le C\int_{\R_{+}^d}\int_0^1
            \frac{e^{\frac{t^2}{4\log r}}}{(-\log r)^{\frac{3}{2}}}\bigg|1+\frac{t^2}{2\log r}\bigg|\bigg(\prod_{j=1}^d\frac{y_j^{\frac{\al_j}{2}-\frac{1}{4}}}{x_j^{\frac{\al_j}{2}+\frac{1}{4}}}
            \frac{e^{-\frac{(\sqrt{rx_j}-\sqrt{y_j})^2}{1-r}}          }{r^{\frac{\al_j}{2}+\frac{1}{4}}\sqrt{1-r}}\bigg)\frac{\mathrm{d}r}{r}\mathrm{d}y\\
            &\quad\le C\int_{0}^1 \frac{e^{\frac{t^2}{4\log r}}}{(-\log r)^{\frac{3}{2}}}\bigg|1+\frac{t^2}{2\log r}\bigg|\bigg(\int_{\R_{+}^d}
            \prod_{j=1}^d\frac{y_j^{\frac{\al_j}{2}-\frac{1}{4}}}{x_j^{\frac{\al_j}{2}+\frac{1}{4}}}e^{-\frac{(\sqrt{rx_j}-\sqrt{y_j})^2}{1-r}}\mathrm{d}y
            \bigg)\frac{(1-r)^{-\frac{d}{2}}}{r^{\frac{|\al|}{2}+\frac{d}{4}}}\frac{\mathrm{d}r}{r}\\
            &\quad\le C\int_{0}^1 \frac{e^{\frac{t^2}{4\log r}}}{(-\log r)^{\frac{3}{2}}}\bigg|1+\frac{t^2}{2\log
            r}\bigg|\frac{\mathrm{d}r}{r}
        \end{align*}
         with $\al\in(-\frac{1}{2},\infty)^d$. Because of (\ref{eq12}) and proceed as before, one has
        \begin{equation*}
            \int_{\R_{+}^d}\bigg|\frac{\partial p_t^{\al}(x,y)}{\partial t}\bigg|\mathrm{d}y\le \frac{C}{t}.
        \end{equation*}
     (iii) If $\frac{\sqrt{e^{-t}x_jy_j}}{1-e^{-t}}\le 1$, $j=1,2,\ldots,k$,
     and $\frac{\sqrt{e^{-t}x_jy_j}}{1-e^{-t}}> 1$ with $\al_j>-\frac{1}{2}$ for $j=k+1,k+2,\ldots,d$, where $k\in\mathbb{N},0\le k\le d$, using the integral representation of Poisson-Laguerre kernel (\ref{eq-12}) and exchanging the order of integration and (\ref{eq-14.1}), we get
     \begin{align*}
        &\int_{\R_{+}^d}\bigg|\frac{\partial p_t^{\al}(x,y)}{\partial t}\bigg|\mathrm{d}y\\
        &\le\frac{1}{2\sqrt{\pi}}\int_{\R_{+}^d}\int_0^1\frac{|1+\frac{t^2}{2\log r}|e^{\frac{t^2}{4\log r}}}{(1-r)^d(-\log r)^{\frac{3}{2}}}\prod_{j=1}^{d}\bigg(\frac{y_j}{x_j}\bigg)^{\frac{\al_j}{2}}
        e^{-\frac{rx_j+y_j}{1-r}-\frac{\al_j\log r}{2}}\bigg|I_{\al_j}
        \bigg(\frac{2\sqrt{rx_jy_j}}{1-r}\bigg)\bigg|\frac{\mathrm{d}r}{r}\mathrm{d}y\\
        &\approx\frac{1}{2\sqrt{\pi}}\int_{\R_{+}^d}\int_0^1\frac{|1+\frac{t^2}{2\log r}|e^{\frac{t^2}{4\log r}}}{(1-r)^d(-\log r)^{\frac{3}{2}}}\prod_{j=1}^{k}
        \bigg(\frac{y_j}{x_j}
        \bigg)^{\frac{\al_j}{2}}e^{-\frac{rx_j+y_j}{1-r}-\frac{\al_j\log r}{2}}
        \bigg(\frac{2\sqrt{rx_jy_j}}{1-r}\bigg)^{\al_{j}}\prod_{j=k+1}^{d}
        \bigg(\frac{y_j}{x_j}\bigg)^{\frac{\al_j}{2}}\\
        &\quad\times e^{-\frac{rx_j+y_j}{1-r}-\frac{\al_j\log r}{2}}\bigg(\frac{2\sqrt{rx_jy_j}}{1-r}\bigg)^{-\frac{1}{2}}
        \bigg(\frac{2\sqrt{rx_jy_j}}{1-r}\bigg)^{\frac{1}{2}}I_{\al_j}
        \bigg(\frac{2\sqrt{rx_jy_j}}{1-r}\bigg)\frac{\mathrm{d}r}{r}\mathrm{d}y\\
        &\le\frac{2^{(\al_1+\al_2+\cdots+\al_k)}\prod_{j=1}^{k}\Gamma(\al_j+1)}{(2\sqrt{\pi})^{d-k+1}}\int_0^1\frac{e^{\frac{t^2}{4\log r}}}{(-\log r)^{\frac{3}{2}}}\bigg|1+\frac{t^2}{2\log r}
        \bigg|\int_{\R_{+}^k}\prod_{j=1}^k
        \frac{e^{-\frac{rx_j+y_j}{1-r}}}{(1-r)^{\al_j+1}}y_j^{\al_j}\mathrm{d}y_1\mathrm{d}y_2\cdots\mathrm{d}y_k\\
        &\quad\times\int_{\R_{+}^{d-k}}\prod_{j=k+1}^{d}
        \frac{y_j^{\frac{\al_j}{2}-\frac{1}{4}}}{x_j^{\frac{\al_j}{2}+\frac{1}{4}}}\frac{e^{-\frac{(\sqrt{rx_j}-\sqrt{y_j})^2}{1-r}}}{r^{\frac{\al_j}{2}+\frac{1}{4}}\sqrt{1-r}}
        \mathrm{d}y_{k+1}\mathrm{d}y_{k+2}\cdots\mathrm{d}y_{d}\frac{\mathrm{d}r}{r}\\
        &\le\frac{2^{(\al_1+\al_2+\cdots+\al_k)}\prod_{j=1}^{k}\Gamma(\al_j+1)}{(2\sqrt{\pi})^{d-k+1}}\int_0^1\frac{e^{\frac{t^2}{4\log r}}}{(-\log r)^{\frac{3}{2}}}\bigg|1+\frac{t^2}{2\log r}\bigg|e^{-\frac{r(x_1+x_2+\cdots+x_k)}{1-r}}\\
        &\quad\times\int_{\R_{+}^{d-k}}\prod_{j=k+1}^{d}
        \frac{y_j^{\frac{\al_j}{2}-\frac{1}{4}}}{x_j^{\frac{\al_j}{2}+\frac{1}{4}}}\frac{e^{-\frac{(\sqrt{rx_j}-\sqrt{y_j})^2}{1-r}}}{r^{\frac{\al_j}{2}+\frac{1}{4}}\sqrt{1-r}}
        \mathrm{d}y_{k+1}\mathrm{d}y_{k+2}\cdots\mathrm{d}y_{d}\frac{\mathrm{d}r}{r}\\
        &\le C\int_0^1\frac{e^{\frac{t^2}{4\log r}}}{(-\log r)^{\frac{3}{2}}}\bigg|1+\frac{t^2}{2\log r}\bigg|\frac{\mathrm{d}r}{r},
     \end{align*}
     where in the last two inequalities we have used (\ref{eq13-1}) and
     (\ref{eq20.1}),
     respectively. Therefore, we apply (\ref{eq12}) to deduce that $\int_{\R_{+}^d}\big|
     \frac{\partial p_t^{\al}(x,y)}{\partial t}\big|\mathrm{d}y\le \frac{C}{t}$, which deduces the  desired proof.

        Finally, we utilize the induction to establish the general case  of (\ref{eq11}).
        Assume that the case for $k=1$ has already been proved and  (\ref{eq11}) is valid for a certain $k$. We  need to verify its validity for
        $k+1$. Via the semigroup property and taking $u=t+s$  we obtain
        \begin{align*}
            \frac{\partial^{k+1}p_u^{\al}(x,y)}{\partial u^{k+1}}&=\frac{\partial}{\partial s}\frac{\partial^k}{\partial t^k}p_{t+s}^{\al}(x,y)\\
            &=\frac{\partial}{\partial s}\frac{\partial^k}{\partial t^k}\bigg[
            \int_{\R_{+}^d}p_s^{\al}(x,v)p_t^{\al}(v,y)\mathrm{d}v\bigg]\\
            &=\int_{\R_{+}^d}\frac{\partial p_{s}^{\al}(x,v)}{\partial s}\frac{\partial^k p_t^{\al}(v,y)}{\partial t^k}\mathrm{d}v.
        \end{align*}
        It follows that
        \begin{align*}
            \int_{\R_{+}^d}\bigg|\frac{\partial^{k+1}p_u^{\al}(x,y)}{\partial u^{k+1}}\bigg|\mathrm{d}y
            &\le\int_{\R_{+}^d}\int_{\R_{+}^d}\bigg|\frac{\partial p_s^{\al}(x,v)}{\partial s}\bigg| \bigg|\frac{\partial^k p_t^{\al}(v,y)}{\partial t^k}\bigg|\mathrm{d}v\mathrm{d}y\\
            &\le \int_{\R_{+}^d}\bigg|\frac{\partial p_s^{\al}(x,v)}{\partial s}  \bigg|\int_{\R_{+}^d}\bigg|\frac{\partial^k p_t^{\al}(v,y)}{\partial t^k}\bigg|
            \mathrm{d}y\mathrm{d}v\\
            &\le \frac{C}{s}\frac{C}{t^k}.
        \end{align*}
        Via taking $s=t=\frac{u}{2}$ we deduce that the case for $k+1$ is proved.
    \end{proof}

    \section{Laguerre Lipschitz spaces}\label{sec-2}
    It is well known that the Poisson semigroup provides an alternative characterization of
    the Lipschitz spaces (see \cite{Stein}). Following this approach and using the Poisson-Hermite semigroup,
    Gatto and Urbina \cite{Gatto} introduced the Gaussian Lipschitz space (see also \cite{GPU} and \cite{PUA}). Similarly, we
     can define the Laguerre Lipschitz space  as follows.
    \begin{definition}\label{def3.1}
        For $\beta>0$, let $n$ be the smallest integer greater than $\beta$.
        A $L^{\infty}$ function defined in $\R_{+}^d$ belongs to the Laguerre Lipschitz space
        $Lip_{\beta}(\mu_{\al})$ associated with $\la$ if there exists a constant $A_{\beta}(f)$ such that
        \begin{equation}\label{eq15}
        \bigg\|\frac{\partial^n P_t^{\al}f}{\partial t^n}\bigg\|_{\infty}\le
        A_{\beta}(f)t^{-n+\beta}.
        \end{equation}
        The norm of $f\in Lip_{\beta}(\mu_{\al})$ is defined as
        \begin{equation*}
        \|f\|_{Lip_{\beta}(\mu_{\al})}:=\|f\|_{\infty}+A_{\beta}(f),
        \end{equation*}
        where $A_{\beta}(f)$ is the smallest constant appearing in (\ref{eq15}).
    \end{definition}
    \begin{remark}\label{remk3.2}
    We only focus on the condition (\ref{eq15}) when $t$ approaches zero, since
    \begin{equation}\label{eq16}
    \bigg\|\frac{\partial^n P_t^{\al}f}{\partial t^n}\bigg\|_{\infty}\le
    At^{-n}
    \end{equation}
    implies (\ref{eq15}) when $t$ away from zero. For the proof of (\ref{eq16}), we conclude that for $f\in L^{\infty}$ and
     $\al\in(-\frac{1}{2},\infty)^d$,
    \begin{equation*}
    \bigg\|\frac{\partial^n P_t^{\al}f}{\partial t^n}\bigg\|_{\infty}\le
    \int_{\R_{+}^d}\bigg|\frac{\partial^n p_t^{\al}(x,y)}{\partial t^n}\bigg| |f(y)|\mathrm{d}y\le
    \frac{C}{t^n}\|f\|_{\infty},
    \end{equation*} where we have used    (\ref{eq11}) in the last step.
    \end{remark}
    Now, we need to establish the following equivalent relations:
    \begin{proposition}\label{pro3.1}
          Let $f\in L^{\infty}$, $\al\in(-\frac{1}{2},\infty)^d$ and $\beta>0$. Then, for any two integers $k$ and $l$ that both greater than $\beta$,
           we conclude that
          \begin{equation}\label{eq-14}
            \bigg\|\frac{\partial^k P_t^{\al}f}{\partial t^k}\bigg\|_{\infty}\le A_{\beta,k}t^{-k+\beta}
          \end{equation}
          and
          \begin{equation}\label{eq-15}
            \bigg\|\frac{\partial^l P_t^{\al}f}{\partial t^l}\bigg\|_{\infty}\le
            A_{\beta,l}t^{-l+\beta}
          \end{equation}
          are equivalent. Furthermore, the smallest constants $A_{\beta,k}$ and $A_{\beta,l}$ that satisfy the aforementioned inequalities are comparable.
    \end{proposition}
     \begin{proof}
        It is sufficient to prove that for $k>\beta$,
        \begin{equation}\label{eq13}
            \bigg\|\frac{\partial^k P_t^{\al}f}{\partial t^k}\bigg\|_{\infty}\le A_{\beta,k}t^{-k+\beta}
        \end{equation}
        and
        \begin{equation}\label{eq14}
            \bigg\|\frac{\partial^{k+1} P_t^{\al}f}{\partial t^{k+1}}\bigg\|_{\infty}\le
            A_{\beta,k+1}t^{-(k+1)+\beta}
        \end{equation}
        are equivalent.

        We first assume that (\ref{eq13}) holds. Then by the semigroup property of
        $\{P_t^{\al}\}_{t\ge 0}$ and Lemma \ref{lem2.1}, we get
        \begin{align*}
            \bigg\|\frac{\partial^{k+1} P_{t}^{\al}f}{\partial t^{k+1}}\bigg\|_{\infty}
            &=\bigg\|\frac{\partial P_{t_1}^{\al}}{\partial t_1}\bigg(\frac{\partial^k P_{t_2}^{\al}f}{\partial t_2^k}\bigg)\bigg\|_{\infty}\\
            &\le \bigg\|\frac{\partial^k P_{t_2}^{\al}f}{\partial t_2^k}\bigg\|_{\infty}\int_{\R_{+}^d}\bigg|\frac{\partial p_{t_1}^{\al}(\cdot,y)}{\partial t_1}\bigg| \mathrm{d}y\\
            &\le A_{\beta,k}t_2^{-k+\beta}Ct_1^{-1},
        \end{align*}
        where $t=t_1+t_2$. Taking $t_1=t_2=t/2$, we conclude that  (\ref{eq14}) is valid.

        Next, we assume that (\ref{eq14}) holds. Applying  Lemma \ref{lem2.1} again, we
        derive
        \begin{equation*}
            \bigg\|\frac{\partial^k P_{t}^{\al}f}{\partial t^k}\bigg\|_{\infty}\le
            \|f\|_{\infty}\int_{\R_{+}^d}\bigg|\frac{\partial^k p_t^{\al}(x,y)}{\partial t^k}\bigg|\mathrm{d}y\le\frac{C}{t^k}\|f\|_{\infty}.
        \end{equation*}
        Therefore,
        \begin{equation*}
            \lim_{t\to\infty} \frac{\partial^k P_t^{\al}f}{\partial t^k}=0
        \end{equation*}
        uniformly, and then
        \begin{equation*}
            \bigg\|\frac{\partial^k P_{t}^{\al}f}{\partial t^k}\bigg\|_{\infty}
            \le\int_{t}^{\infty}\bigg\|\frac{\partial^{k+1} P_{s}^{\al}f}{\partial s^{k+1}}\bigg\|_{\infty}\mathrm{d}s=Ct^{-k+\beta},
        \end{equation*}
        which implies that (\ref{eq13}) holds and the proof of Proposition \ref{pro3.1} is completed.
     \end{proof}

    \begin{remark}
    Proposition \ref{pro3.1} implies that  the definition of $Lip_{\beta}(\mu_{\al})$ doesn't depend on the index $k$ for $k>\beta$   and the resulting
        norms are equivalent provided that $\al\in (-1/2,\infty)^d$.
    \end{remark}

   At the same time, he Laguerre Lipschitz spaces have the following  monotonicity property, which can be similarly proved by adopting the method of Proposition 2.2 in \cite{Gatto}.
    \begin{proposition}\label{pro3.2}
    Let $\al\in(-\frac{1}{2},\infty)^d$.
    If $0<\beta_1<\beta_2$, then $Lip_{\beta_2}(\mu_{\al})\subset Lip_{\beta_1}(\mu_{\al})$.
    \end{proposition}

    %\begin{proof}
        %Assume that $f\in Lip_{\beta_2}(\mu_{\al})$ and $n\ge \beta_2$, by Definition \ref{def3.1} we %obtain
        %\begin{equation*}
         %   \bigg\|\frac{\partial^n P_t^{\al}f}{\partial t^n}\bigg\|_{\infty}\le
          %  A_{\beta}(f)t^{-n+\beta_2}.
        %\end{equation*}
        %Next, we divided $t>0$ into two parts. \\
        %Case 1: If $0<t<1$, then $t^{-n+\beta_2}\le t^{-n+\beta_1}$.
        %Thus,
        %\begin{equation*}
        %    \bigg\|\frac{\partial^n P_t^{\al}f}{\partial t^n}\bigg\|_{\infty}\le
        %    A_{\beta}(f)t^{-n+\beta_1}.
        %\end{equation*}
        %Case 2: If $t\ge 1$, by (\ref{eq16}) we obtain
        %\begin{equation*}
        %    \bigg\|\frac{\partial^n P_t^{\al}f}{\partial t^n}\bigg\|_{\infty}\le
        %    A_{\beta}(f)t^{-n}.
        %\end{equation*}
        %Since $t^{-n+\beta_1}>t^{-n}$, we also have
        %\begin{equation*}
        %    \bigg\|\frac{\partial^n P_t^{\al}f}{\partial t^n}\bigg\|_{\infty}\le
        %   A_{\beta}(f)t^{-n+\beta_1}.
        %\end{equation*}
        %As $n>\beta_1$, it follows that $f\in Lip_{\beta_1}(\mu_{\al})$.
    %\end{proof}
    \begin{proposition}\label{pro3.3}
        Let $f\in Lip_{\beta}(\mu_{\al})$ and $0<\beta<1$, we have
        \begin{equation}\label{eq18.1}
           \|P_t^{\al}f-f\|_{\infty}\le A_{\beta}(f)t^{\beta}.
        \end{equation}
    \end{proposition}
    \begin{proof}
Using the fundamental theorem of Calculus, we get
        \begin{align*}
          \|P_t^{\al}f-f\|_{\infty}=\bigg\|\int_0^t\frac{\partial P_s^{\al}f}{\partial s}\mathrm{d}s\bigg\|_{\infty}\le\int_0^t \bigg\|\frac{\partial P_s^{\al}f}{\partial s}\bigg\|_{\infty}\mathrm{d}s
          \le A_{\beta}(f)\int_0^t s^{-1+\beta}\mathrm{d}s=A_{\beta}(f)t^{\beta}.
        \end{align*}
    \end{proof}
    \section{Boundedness of fractional integrals and fractional derivatives associated with Laguerre expansions on $Lip_{\beta}(\mu_{\al})$}\label{sec-3}
    The Bessel Laguerre potential of order $\lambda>0$ denoted by $\mathcal{J}_{\lambda}^{\al}$,
    associated to the Laguerre measure, is defined formally as
    \begin{equation*}
        \mathcal{J}_{\lambda}^{\al}=(I+\la)^{-\frac{\lambda}{2}}.
    \end{equation*}
    By (\ref{eqq1}) we have
    \begin{equation*}
       \mathcal{J}_{\lambda}^{\al}L_k^{\al}(x)=\frac{1}{(1+|k|)^{\frac{\lambda}{2}}}L_k^{\al}(x)
    \end{equation*}  for any Laguerre polynomial $L_k^{\al}(x)$.

    Through the principle of linearity, this definition can be broadened to encompass any polynomial, and
    P.A. Meyer's theorem enables us to extend Bessel Laguerre potentials to form a continuous operator
     on $L^p(\R_{+}^d,d\mu_{\al})$ for $1<p<\infty$ (see \cite[Lemma 6.1]{Graczyk}).
    Bessel Laguerre potentials can be defined in another form as an alternative
    integral representation:
    \begin{equation}\label{eq17}
        \mathcal{J}_{\lambda}^{\al}f(x)=\frac{1}{\Gamma(\lambda)}
        \int_0^{\infty}s^{\lambda-1}
        e^{-s}P_s^{\al}f(x)\mathrm{d}s,\ f\in L^p(\R_{+}^d,\mathrm{d}\mu_{\al}).
    \end{equation}
 Please refer to \cite{Graczyk} for more details.

    In what follows, we give the action of the Bessel Laguerre potential on Laguerre Lipschitz spaces.
    \begin{theorem}\label{thm3.1}
    Let $\al\in(-\frac{1}{2},\infty)^d$, $\beta>0$ and $\lambda>0$. Then $\mathcal{J}_{\lambda}^{\al}$ is bounded from
    $Lip_{\beta}(\mu_{\al})$ to $Lip_{\beta+\lambda}(\mu_{\al})$.
    \end{theorem}
    \begin{proof}
        Assume that $f\in Lip_{\beta}(\mu_{\al})$. Then for any $n>\beta+\lambda$,
        by Proposition \ref{pro3.1}  we have
        \begin{equation*}
            \bigg\|\frac{\partial^n P_t^{\al}f}{\partial t^n}\bigg\|_{\infty}\le
            A_{\beta}(f)t^{-n+\beta},\ t>0.
        \end{equation*}

         Since $\|P_{s}^{\al}f\|_{\infty}\le \|f\|_{\infty}$, we have $P_{s}^{\al}f\in L^{\infty}$  due to $f\in L^{\infty}$.  Moreover, by (\ref{eq17}) we get
        \begin{equation}\label{eq18}
            P_t^{\al}(\mathcal{J}_{\lambda}^{\al}f)(x)=\frac{1}{\Gamma(\lambda)}\int_0^{\infty}s^{\lambda-1}e^{-s}P_{t+s}^{\al}f(x)\mathrm{d}s.
        \end{equation}
        Therefore
        \begin{equation}\label{eq29.2}
            \|P_t^{\al}(\mathcal{J}_{\lambda}^{\al}f)\|_{\infty}\le \|f\|_{\infty},
        \end{equation}
   which implies that  $P_t^{\al}(\mathcal{J}_{\lambda}^{\al}f)\in L^{\infty}$.

        By (\ref{eq18}), we  obtain
    \begin{align*}
        \frac{\partial^n P_t^{\al}(\mathcal{J}_{\lambda}^{\al}f)(x)}{\partial t^n}&=\frac{1}{\Gamma(\lambda)}\int_0^{\infty}s^{\lambda-1}e^{-s}\frac{\partial^n P_{t+s}^{\al}f(x)}{\partial t^n}\mathrm{d}s\\
        &=\frac{1}{\Gamma(\lambda)}\int_0^{\infty}s^{\lambda-1}e^{-s}\frac{\partial^n P_{t+s}^{\al}f(x)}{\partial (t+s)^n}\mathrm{d}s.
    \end{align*}
    Then
    \begin{align*}
        \bigg\|\frac{\partial^n P_t^{\al}(\mathcal{J}_{\lambda}^{\al}f)(x)}{\partial t^n}\bigg\|_{\infty}
        &\le \frac{1}{\Gamma(\lambda)}
        \int_0^{t}s^{\lambda-1}e^{-s}\bigg\|\frac{\partial^n P_{t+s}^{\al}f(x)}{\partial (t+s)^n}\bigg\|_{\infty}\mathrm{d}s\\
        &\quad+\frac{1}{\Gamma(\lambda)}\int_t^{\infty}s^{\lambda-1}e^{-s}\bigg\|\frac{\partial^n P_{t+s}^{\al}f(x)}{\partial (t+s)^n}\bigg\|_{\infty}\mathrm{d}s\\
        &=:I+II.
    \end{align*}
    For $I$, we can get
    \begin{align*}
        I&\le\frac{A_{\beta}(f)}{\Gamma(\lambda)}\int_0^t s^{\lambda-1}(t+s)^{-n+\beta}e^{-s}\mathrm{d}s\\
        &\le \frac{A_{\beta}(f)}{\Gamma(\lambda)}t^{-n+\beta}\int_0^t s^{\lambda-1}\mathrm{d}s\\
        &\le Ct^{-n+\beta+\lambda}A_{\beta}(f).
    \end{align*}
    Since $n>\beta+\lambda$, we have
    \begin{align*}
        II&\le\frac{A_{\beta}(f)}{\Gamma(\lambda)}\int_t^{\infty} s^{\lambda-1}e^{-s}(t+s)^{-n+\beta}\mathrm{d}s\\
        &\le\frac{A_{\beta}(f)}{\Gamma(\lambda)}\int_t^{\infty} s^{\lambda-1}e^{-s}s^{-n+\beta}\mathrm{d}s\\
        &\le \int_t^{\infty}s^{-n+\beta+\lambda-1}\mathrm{d}s\\
        &=CA_{\beta}(f)t^{-n+\beta+\lambda}.
    \end{align*}
Hence,
    \begin{equation*}
        \bigg\|\frac{\partial^n P_t^{\al}(\mathcal{J}_{\lambda}^{\al}f)(x)}{\partial t^n}\bigg\|_{\infty}\le CA_{\beta}(f)t^{-n+\beta+\lambda}.
    \end{equation*}
    Therefore,  $\mathcal{J}_{\lambda}^{\al}f\in Lip_{\beta+\lambda}(\mu_{\al})$ and
    via (\ref{eq29.2}) we have
    \begin{align*}
        \|\mathcal{J}_{\lambda}^{\al}f\|_{Lip_{\beta+\lambda}(\mu_{\al})}
        \le C\|f\|_{Lip_{\beta}(\mu_{\al})}.
    \end{align*}
    \end{proof}

    For $f\in L^2(\R_{+}^d,\mathrm{d}\mu_{\al})$ and $\lambda>0$, we define the fractional integral of order $\lambda$
    associated with Laguerre expansions of type $\al$ by:
    \begin{equation*}
        I_{\lambda}^{\al}=(\la)^{-\frac{\lambda}{2}}\Pi_0,
    \end{equation*}
    where
     $$\Pi_0f=f-\int_{\R_{+}^d}f(y)\mathrm{d}\mu_{\al}.$$
    This   definition is correct for all polynomials due to the fact
    that
    \begin{equation*}
        I_{\lambda}^{\al}L_k^{\al}=|k|^{-\frac{\lambda}{2}}L_k^{\al},\
        |k|>0,
    \end{equation*} for all Laguerre
    polynomials.
    If $f$ is a polynomial with $\int_{\R_{+}^d}f(y)\mathrm{d}\mu_{\al}(y)=0$,
    then
    \begin{equation}\label{eq19}
        I_{\lambda}^{\al}f(x)=\frac{1}{\Gamma(\lambda)}\int_0^{\infty}s^{\lambda-1}
        P_s^{\al}f(x)\mathrm{d}s.
    \end{equation}

    By P.A. Meyer's multiplier theorem, $I_{\lambda}^{\al}$ satisfies a continuous     extension to $L^p(\R_{+}^d,\mathrm{d}\mu_{\al})$ for $1<p<\infty$, and (\ref{eq19})
    can be extended for $f\in L^p(\R_{+}^d,\mathrm{d}\mu_{\al})$ (see \cite{Graczyk}).

  %  The fractional integral of order $\lambda$ associated with Laguerre expansions is not a bounded operator
  %on $L^{\infty}$ nor on $Lip_{\beta}(\mu_{\al})$. Please see \cite{Gatto} for the case of the classical Gaussian setting.

    The Laguerre fractional derivative $D_{\lambda}^{\al}$ of order $\lambda>0$
    is defined formally by:
    \begin{equation*}
        D_{\lambda}^{\al}=(\la)^{\frac{\lambda}{2}},
    \end{equation*}
therefore,  for all Laguerre polynomials we have
    \begin{equation*}
            D_{\lambda}^{\al}L_k^{\al}=|k|^{\frac{\lambda}{2}}L_k^{\al}.
    \end{equation*}
    For $0<\lambda<\beta<1$ and $f\in Lip_{\beta}(\mu_{\al})$, we have an alternative
    representation of $D_{\lambda}^{\al}$ as follows:
    \begin{equation}\label{eq20}
        D_{\lambda}^{\al}f=\frac{1}{c_{\lambda}}\int_0^{\infty}s^{-\lambda-1}(P_s^{\al}-I)f\mathrm{d}s,
    \end{equation}
    where $c_{\lambda}=\int_0^{\infty}u^{-\lambda-1}(e^{-u}-1)\mathrm{d}u<\infty$ when $0<\lambda<1$.

    Now we prove the following theorem.
    \begin{theorem}\label{thm4.2}
    Let $\al\in(-\frac{1}{2},\infty)^d$ and $0<\lambda<\beta<1$,
    the Laguerre fractional derivative of order $\lambda$, denoted as $D_{\lambda}^{\al}$,
    is bounded from $Lip_{\beta}(\mu_{\al})$ to $ Lip_{\beta-\lambda}(\mu_{\al})$.
    \end{theorem}
    \begin{proof}
        Assume that $f\in Lip_{\beta}(\mu_{\al})$, then $f\in L^{\infty}$ and $\|\frac{\partial P_t^{\al}f}{\partial t}\|_{\infty}\le A_{\beta}(f)t^{-1+\beta}$.
        By (\ref{eq20}) and Proposition \ref{pro3.3} we have
        \begin{align*}
            |D_{\lambda}^{\al}f(x)|&\le\frac{1}{c_{\lambda}}\int_0^1 s^{-\lambda-1}
            |P_s^{\al}f(x)-f(x)|\mathrm{d}s+\frac{1}{c_{\lambda}}\int_1^{\infty} s^{-\lambda-1}|P_s^{\al}f(x)-f(x)|\mathrm{d}s\\
            &\le \frac{1}{c_{\lambda}}\int_0^1 s^{-\lambda-1}
            \|P_s^{\al}f-f\|_{\infty}\mathrm{d}s+\frac{2\|f\|_{\infty}}{c_{\lambda}}
            \int_1^{\infty}s^{-\lambda-1}\mathrm{d}s\\
            &\le \frac{A_{\beta}(f)}{c_{\lambda}}\int_0^1 s^{\beta-\lambda
            -1}\mathrm{d}s+\frac{2\|f\|_{\infty}}{c_{\lambda}}
            \int_1^{\infty}s^{-\lambda-1}\mathrm{d}s\\
            &=\frac{A_{\beta}(f)}{c_{\lambda}(\beta-\lambda)}+\frac{2\|f\|_{\infty}}{\lambda c_{\lambda}}\\
            &\le C_{\beta,\lambda}\|f\|_{Lip_{\beta}(\mu_{\al})},
        \end{align*}
        which shows that $D_{\lambda}^{\al}f\in L^{\infty}$.

        To prove the Lipschitz condition, fixing $t$
        and using again (\ref{eq20}), we obtain
        \begin{align*}
            \frac{\partial}{\partial t}(P_t^{\al}D_{\lambda}^{\al}f(x))&=\frac{1}{c_{\lambda}}\frac{\partial}{\partial t}\bigg[\int_0^{\infty}s^{-\lambda-1}
            (P_{t+s}^{\al}f(x)-P_t^{\al}f(x))\mathrm{d}s\bigg]\\
            &=\frac{1}{c_{\lambda}}\int_0^{\infty}s^{-\lambda-1}\bigg[\frac{\partial P_{s+t}^{\al}f(x)}{\partial t}-\frac{\partial P_{t}^{\al}f(x)}{\partial t}\bigg]\mathrm{d}s\\
            &=\frac{1}{c_{\lambda}}\int_0^{t}s^{-\lambda-1}\bigg[\frac{\partial P_{s+t}^{\al}f(x)}{\partial t}-\frac{\partial P_{t}^{\al}f(x)}{\partial t}\bigg]\mathrm{d}s\\
            &\quad+\frac{1}{c_{\lambda}}
            \int_t^{\infty}s^{-\lambda-1}\bigg[\frac{\partial P_{s+t}^{\al}f(x)}{\partial t}-\frac{\partial P_{t}^{\al}f(x)}{\partial t}\bigg]\mathrm{d}s\\
            &=:I+II.
        \end{align*}
        We use Proposition \ref{pro3.1} to get
        \begin{equation*}
            \bigg\|\frac{\partial^2 P_r^{\al}f}{\partial r^2}\bigg\|_{\infty}\le Ar^{\beta-2},
        \end{equation*}
        and by the fundamental theorem of calculus, we write
        \begin{equation}\label{eq23.1}
        \begin{split}
        \bigg|\frac{\partial P_{t+s}^{\al}f(x)}{\partial t}-\frac{\partial P_t^{\al}f(x)}{\partial t}
        \bigg|\le \int_{t}^{t+s}\bigg\|\frac{\partial^2 P_r^{\al}f(x)}{\partial r^2}\bigg\|_{\infty}\mathrm{d}r\le At^{\beta-2}s
        \end{split}
        \end{equation}
        for $s>0$.\\
        Hence,
        \begin{align*}
            |I|&\le\frac{1}{c_{\lambda}}\int_0^{t}s^{-\lambda-1}\bigg|\frac{\partial P_{t+s}^{\al}f(x)}{\partial t}-\frac{\partial P_t^{\al}f(x)}{\partial t}\bigg|\mathrm{d}s\\
            &\le A\frac{t^{\beta-2}}{c_{\lambda}}\int_0^t s^{-\lambda}\mathrm{d}s\\
            &=C_{\beta,\lambda}t^{-1+\beta-\lambda}.
        \end{align*}
        Next, we deal with $II$. Since $f\in Lip_{\beta}(\mu_{\al})$,
        therefore
        \begin{align*}
            |II|&\le \frac{1}{c_{\lambda}}\int_{t}^{\infty}s^{-\lambda-1}\bigg[
            \bigg|\frac{\partial P_{t+s}^{\al}f(x)}{\partial t}\bigg|+\bigg|\frac{\partial P_{t}^{\al}f(x)}{\partial t}\bigg|\bigg]
            \mathrm{d}s\\
            &\le \frac{C}{c_{\lambda}}\int_t^{\infty}
            s^{-\lambda-1}[(t+s)^{-1+\beta}+t^{-1+\beta}]\mathrm{d}s\\
            &\le Ct^{-1+\beta}\int_t^{\infty}s^{-\lambda-1}\mathrm{d}s\\
            &=C_{\beta,\lambda}t^{-1+\beta-\lambda}.
        \end{align*}
        Then we have
        \begin{equation*}
            \bigg\|\frac{\partial}{\partial t}(P_t^{\al}D_{\lambda}^{\al}f)\bigg\|_{\infty}
            \le C_{\beta,\lambda}t^{\beta-\lambda-1},
        \end{equation*}
        This finishes the proof.
    \end{proof}

    The Bessel fractional derivative with respect to the Laguerre measure denoted by $\mathcal{D}_{\lambda}^{\al}$ (the Bessel Laguerre fractional derivative in short) is defined formally as
    \begin{equation*}
        \bfd=(I+\sqrt{\la})^{\lambda},
    \end{equation*}
    which means that
    \begin{equation*}
        \bfd L_k^{\al}(x)=(1+\sqrt{|k|})^{\lambda}L_k^{\al}(x)
    \end{equation*} for all Laguerre polynomials.
    If $0<\lambda<1$, $\bfd f$ has the following integral representation:
    \begin{equation*}\label{eq22}
        \bfd f=\frac{1}{c_{\lambda}}\int_0^{\infty}
        t^{-\lambda-1}(e^{-t}P_t^{\al}-I)f\mathrm{d}t,
    \end{equation*}
    where $c_{\lambda}=\int_0^{\infty}u^{-\lambda-1}(e^{-u}-1)\mathrm{d}u<\infty$.

    Similarly, we  can investigate the action of the Bessel Laguerre fractional derivative $\bfd$ on the Laguerre Lipschitz space.
    \begin{theorem}\label{thm3.3}
     Let $\al\in(-\frac{1}{2},\infty)^d$ and $0<\lambda<\beta<1$. Then the Bessel Laguerre fractional derivative $\bfd$ is bounded from $Lip_{\beta}(\mu_{\al})$ into
     $Lip_{\beta-\lambda}(\mu_{\al})$.
    \end{theorem}
    \begin{proof}
    This theorem can be proved by a method similar to the one used in the proof of Theorem \ref{thm4.2}. We leave the details to the reader.
    \end{proof}

    Furthermore, if $\lambda\ge 1$, let $k$ be the smallest integer such that $\lambda<k$, it can be proved that the Laguerre fractional derivative $D_{\lambda}^{\al}$ can be represented as
    \begin{equation}\label{eq23}
        D_{\lambda}^{\al}f=\frac{1}{c_{\lambda}^k}\int_0^{\infty}s^{-\lambda-1}(P_s^{\al}-I)^k f\mathrm{d}s,
    \end{equation}
    and the Bessel Laguerre fractional derivative $\bfd$ can be represented as
    \begin{equation}\label{eq24}
        \bfd f=\frac{1}{c_{\lambda}^k}\int_0^{\infty}
        s^{-\lambda-1}(e^{-s}P_s^{\al}-I)^kf\mathrm{d}s,
    \end{equation}
where $c_{\lambda}^k=\int_0^{\infty}u^{-\lambda-1}(e^{-u}-1)^k
    \mathrm{d}u<\infty$.

    Since it is obvious to verify that for all Laguerre polynomials $L_k^{\al}$,
    $$D_{\lambda}^{\al}L_k^{\al}=|k|^{\frac{\lambda}{2}}L_k^{\al}\ \mathrm{and}\
    \bfd L_k^{\al}=(1+\sqrt{|k|})^{\lambda}L_k^{\al},$$ then we conclude that (\ref{eq23}) and (\ref{eq24}) are valid.

    In what follows, we will study the boundedness of $D_{\lambda}^{\al}$ and $\bfd$ on the Laguerre Lipschitz space  when $\lambda\ge 1$.
    \begin{theorem}\label{thm4.4}
    Let $\al\in(-\frac{1}{2},\infty)^d$. For $1\le \lambda<\beta$, then
    \item{\rm (i)} The Laguerre fractional derivative of order $\lambda$, denoted as $D_{\lambda}^{\al}$, is bounded from $Lip_{\beta}(\mu_{\al})$ into $Lip_{\beta-\lambda}(\mu_{\al})$.
    \item{\rm (ii)} The Bessel Laguerre fractional derivative of order $\lambda$, denoted as $\bfd$ is bounded from $Lip_{\beta}(\mu_{\al})$ into $ Lip_{\beta-\lambda}(\mu_{\al})$.
    \end{theorem}

    Before giving the proof of Theorem \ref{thm4.4}, we need some preparations.
    By the binomial theorem and the semigroup property, we get
    \begin{equation}\label{eq27}
     \begin{split}
     (P_t^{\al}-I)^kf(x)&=\sum_{j=0}^k \binom{k}{j}(P_t^{\al})^{k-j}(-I)^{j}f(x)
     =\sum_{j=0}^{k}\binom{k}{j}(-1)^j P_{(k-j)t}^{\al}f(x)\\
     &=\sum_{j=0}^{k}\binom{k}{j}(-1)^j u(x,(k-j)t)=\Delta_t^{k}(u(x,\cdot),0),
     \end{split}
    \end{equation}
    where $u(x,t)=P_t^{\al}f(x)$, and we call
    \begin{equation}\label{eq28-1}
        \Delta_s^k(f,t)=\sum_{j=0}^k\binom{k}{j}(-1)^j f(t+(k-j)s)
    \end{equation}
    the $k$-th order forward difference of $f$ starting at $t$ with the increment $s$.
    Next, we will need certain technical outcomes regarding
     forward differences that will be employed later. For more established results in the theory of forward differences, refer to \cite{Tfo}, while the proofs can be found in \cite{Gatto}.
    \begin{lemma}
        For any positive integer $k$ and $j$, then
        \item{\rm (i)}
        \begin{equation*}
            \Delta_{s}^k(f,t)=\Delta_s^{k-1}(\Delta_s(f,\cdot),t)=\Delta_s(\Delta_s^{k-1}(f,\cdot),t).
        \end{equation*}
        \item{\rm (ii)}
        \begin{equation*}
            \Delta_{s}^k(f,t)=\int_t^{t+s}\int_{v_1}^{v_1+s}\cdots \int_{v_{k-2}}^{v_{k-2}+s}\int_{v_{k-1}}^{v_{k-1}+s}f^{(k)}(v_k)\mathrm{d}v_k\mathrm{d}v_{k-1}\cdots\mathrm{d}v_2\mathrm{d}v_1.
        \end{equation*}
        \item{\rm (iii)}
        \begin{equation*}\label{eq28.2}
            \frac{\partial}{\partial s}(\Delta_s^k(f,t))=k\Delta_s^{k-1}({f}^{\prime},t+s),
        \end{equation*}
        \begin{equation}\label{eq29.1}
            \frac{\partial^j}{\partial t^j}(\Delta_s^k(f,t))=\Delta_s^k(f^{(j)},t).
        \end{equation}
    \end{lemma}

    The following proposition provides estimates similar to (\ref{eq18.1}) and (\ref{eq23.1}), which are crucial for proving Theorem \ref{thm4.4}.
    \begin{proposition}(see \cite{Gatto})\label{prop3.5}
        Let $\delta\in \R$ and $k\in \mathbb{N}_{+}$ such that $\delta<k$. For some integer $k$, if a function $f$ satisfies
        \begin{equation}\label{eq28.1}
            |f^{(k)}(r)|\le Cr^{-k+\delta},
        \end{equation}
        then
        \begin{equation}\label{eq31}
            |\Delta_s^k(f,t)|\le Cs^k t^{-k+\delta}.
        \end{equation}
    \end{proposition}

In the following statement,   Proposition \ref{pro3.3} is
generalized to the multidimensional case.
        \begin{proposition}(see \cite{Gatto})
        \item{\rm (i)} If $f\in L^{\infty}$, for any positive integer
        $k$,  then
        \begin{equation}\label{eq28}
            \|(P_t^{\al}-I)^k f\|_{\infty}\le 2^k\|f\|_{\infty}.
        \end{equation}
        \item{\rm (ii)} If $f\in Lip_{\beta}(\mu_{\al})$ with  $\beta>1$ and let $n$ be the smallest integer bigger than $\beta$, then
        \begin{equation}\label{eq29}
            \|(P_t^{\al}-I)^n f\|_{\infty}\le A_{\beta}(f)t^{\beta}.
        \end{equation}
        \end{proposition}
        %\begin{proof}
        %\item{\rm (i)} By (\ref{eq27}) we obtain
        %\begin{equation*}
        %    (P_t^{\al}-I)^k f(x)=\sum_{j=0}^{k}\binom{k}{j}(-1)^j P_{(k-j)t}^{\al}f(x).
        %\end{equation*}
        %Consequently, for any positive integer $k$,
        %\begin{equation*}
        %    \|(P_t^{\al}-I)^kf\|_{\infty}\le \sum_{j=0}^{k}\binom{k}{j}\|
        %    P_{(k-j)t}^{\al}f\|_{\infty}=2^k\|f\|_{\infty},
        %\end{equation*}which implies that  (\ref{eq28}) is valid.
        %\item{\rm (ii)} Now we are in a position to prove (\ref{eq29}). Noting that $\beta<n$ and %using (\ref{eq15}), we have
        %\begin{equation*}
        %    \bigg\|\frac{\partial^{n-1}}{\partial %t^{n-1}}(u^{\prime}(\cdot,t))\bigg\|_{\infty}=\bigg\|\frac{\partial^n}{\partial %t^n}(u(\cdot,t))\bigg\|_{\infty}\le A_{\beta}(f)t^{-n+1+(\beta-1)},
        %\end{equation*}
        %as a result, condition (\ref{eq28.1}) is satisfied for $\delta=\beta-1$. Then by (\ref{eq27}), %(\ref{eq28.2}) and (\ref{eq31}) with $t=s=r$,
        %\begin{align*}
        %    |(P_t^{\al}-I)^n f(x)|&\le \int_0^t\bigg|\frac{\partial }{\partial r}  %(\Delta_r^n(u(x,\cdot),0))\bigg|\mathrm{d}r\\
        %    &=n\int_0^t \bigg|\Delta_r^{n-1}(u^{\prime}(x,\cdot),r)\bigg|\mathrm{d}r\\
        %    &\le n A_{\beta}(f)\int_0^t r^{n-1}r^{-n+1+(\beta-1)}\mathrm{d}r=Ct^{\beta}.
        %\end{align*}
        %\end{proof}

        Based on the above results, we are now in a position to prove   Theorem \ref{thm4.4}.
        \begin{proof}{\rm (i)} Taking $f\in Lip_{\beta}(\mu_{\al})$ and by
        Definition \ref{def3.1}, we get $f\in L^{\infty}$
        and $$\bigg\|\frac{\partial^n u(\cdot,t)}{\partial t^n}\bigg\|_{\infty}\le
        A_{\beta}(f)t^{-n+\beta}.$$
        Since $1\le\lambda<\beta$, let $k$ be the smallest integer bigger than $\lambda$ and let $n$ be the smallest integer bigger than $\beta$ so that  $k\le n$. Using (\ref{eq23}), we obtain
        \begin{align*}
            |D_{\lambda}^{\al}f(x)|&\le \frac{1}{c_{\lambda}^k}\int_0^{\infty}
            s^{-\lambda-1}|(P_s^{\al}-I)^k f(x)|\mathrm{d}s\\
            &=\frac{1}{c_{\lambda}^k}\int_0^1s^{-\lambda-1}|(P_s^{\al}-I)^k f(x)|\mathrm{d}s+\frac{1}{c_{\lambda}^k}\int_1^{\infty}s^{-\lambda-1}|(P_s^{\al}-I)^k f(x)|\mathrm{d}s\\
            &=:I+II.
        \end{align*}
        When $k=n$, the argument is obvious, thus we will assume $k<n$.
        Choose
        $\epsilon>0$ such that $\lambda+\epsilon<k$. By Proposition \ref{pro3.2}
        we know $Lip_{\beta}(\mu_{\al})\subset Lip_{\lambda+\epsilon}(\mu_{\al})$.
        Then
        (\ref{eq29}) gives
        \begin{equation*}
            I\le \frac{1}{c_{\lambda}^k}\int_0^1 s^{-\lambda-1}\|(P_s^{\al}-I)^k f(x)\|_{\infty}\mathrm{d}s=\frac{A_{\lambda+\epsilon}(f)}{c_{\lambda}^k\epsilon}.
        \end{equation*}
        On the other hand, by   (\ref{eq28}) we have
        \begin{equation*}
            II\le \frac{1}{c_{\lambda}^k}\int_1^{\infty} s^{-\lambda-1}\|(P_s^{\al}-I)^k f(x)\|_{\infty}\mathrm{d}s\le
            \frac{2^k\|f\|_{\infty}}{c_{\lambda}^k}\int_1^{\infty}s^{-\lambda-1}\mathrm{d}s=C_{\lambda}\|f\|_{\infty}.
        \end{equation*}
        It follows that $D_{\lambda}^{\al}f\in L^{\infty}$.

      Finally, we need to prove the Lipschitz condition. Noting that by (\ref{eq27}),
      (\ref{eq28-1}) and taking $u(x,t)=P_t^{\al}f(x)$, we have the following semigroup
       property:
        \begin{align*}
            P_t^{\al}[(P_s^{\al}-I)^kf(x)]&=P_t^{\al}(\Delta_s^k(u(x,\cdot),0))=
            P_t^{\al}\bigg(\sum_{j=0}^{k}\binom{k}{j}(-1)^jP_{(k-j)s}^{\al}f(x)\bigg)\\
            &=\sum_{j=0}^{k}\binom{k}{j}(-1)^{j}P_{t+(k-j)s}^{\al}f(x)
            =\Delta_s^{k}(u(x,\cdot),t).
        \end{align*}
        For fixed $t>0$, using again (\ref{eq23}) and (\ref{eq29.1}), we
        get
        \begin{align*}
             \frac{\partial^n (P_t^{\al}D_{\lambda}^{\al}f(x))}{\partial t^n}&=\frac{1}{c_{\lambda}^k}\frac{\partial^n}{\partial t^n}\bigg[  \int_0^{\infty}s^{-\lambda-1}P_t^{\al}[(P_s-I)^kf(x)]
             \mathrm{d}s\bigg]\\
             &=\frac{1}{c_{\lambda}^k}\int_0^{\infty}s^{-\lambda-1}\frac{\partial^n}{\partial t^n}[\Delta_s^k(u(x,\cdot),t)]\mathrm{d}s\\
             &=\frac{1}{c_{\lambda}^k}\int_0^{t}s^{-\lambda-1}[\Delta_s^k(u^{(n)}(x,\cdot),t)]\mathrm{d}s\\
             &\quad+\frac{1}{c_{\lambda}^k}\int_t^{\infty}s^{-\lambda-1}[\Delta_s^k(u^{(n)}(x,\cdot),t)]\mathrm{d}s\\
             &=:III+IV.
        \end{align*}
        Via (\ref{eq15}) and  Proposition \ref{pro3.1}, we obtain
        \begin{equation*}
            \bigg\|\frac{\partial^k}{\partial t^k}(u^{(n)}(\cdot,t))\bigg\|_{\infty}=\bigg\|\frac{\partial^{n+k}(u(\cdot,t))}{\partial t^{n+k}}\bigg\|_{\infty}\le A_{\beta}(f)t^{-(n+k)+\beta}=A_{\beta}(f)t^{-k+(\beta-n)},
        \end{equation*}
        that is, (\ref{eq28.1}) is satisfied for
        $\delta=\beta-n$. Then by (\ref{eq31}), we obtain
        \begin{align*}
            |III|&\le \frac{1}{c_{\lambda}^k}\int_0^t s^{-\lambda-1}|\Delta_s^k(u^{(n)}(x,\cdot),t)|\mathrm{d}s\\
            &\le \frac{At^{-k+(\beta-n)}}{c_{\lambda}^k}\int_0^t s^{-\lambda+k-1}\mathrm{d}s=C_{\beta,\lambda}t^{-n+\beta-\lambda}.
        \end{align*}
        On the other hand, using (\ref{eq28-1}) and Definition \ref{def3.1}
        \begin{align*}
            |\Delta_s^k(u^{(n)}(x,\cdot),t)|&\le \sum_{j=0}^k\binom{k}{j}|u^{(n)}(x,t+(k-j)s)|\\
            &\le A_{\beta}(f)\sum_{j=0}^k\binom{k}{j}|(t+(k-j)s)^{-n+\beta}|\le
            Ct^{-n+\beta},
        \end{align*}
        where $u(x,t)=P_t^{\al}f(x)$.
        Thus, we further have
        \begin{align*}
            |IV|&\le \frac{1}{c_{\lambda}^k}\int_t^{\infty}
            s^{-\lambda-1}|\Delta_s^k(u^{(n)}(x,\cdot),t)|\mathrm{d}s\\
            &\le \frac{Ct^{-n+\beta}}{c_{\lambda}^k}
            \int_t^{\infty}s^{-\lambda-1}\mathrm{d}s=C_{\beta,\lambda}t^{-n+\beta-\lambda}.
        \end{align*}
        Hence,
        \begin{equation*}
            \bigg\|\frac{\partial^n(P_t^{\al}D_{\lambda}^{\al}f)}{\partial t^n}  \bigg\|_{\infty}\le
            Ct^{-n+\beta-\lambda}.
        \end{equation*}
        Observing that $\beta-\lambda<n$, we conclude from Proposition \ref{pro3.1} that
        $D_{\lambda}^{\al}f\in Lip_{\beta-\lambda}(\mu_{\al})$.
        \item{\rm (ii)} Take $f\in Lip_{\beta}(\mu_{\al})$, that is,  $f\in L^{\infty}$ and
        $\|\frac{\partial^n u(\cdot,t)}{\partial t^n}\|_{\infty}\le A_{\beta}(f)t^{-n+\beta}$. Since $\lambda<\beta$, let $k$ be the smallest integer bigger than
        $\lambda$ and let $n$ be the smallest integer bigger than
        $\beta$. Using (\ref{eq24}), we have
        \begin{align*}
            |\bfd f(x)|&\le\frac{1}{c_{\lambda}^k}\int_0^{\infty}
            s^{-\lambda-1}|(e^{-s}P_s^{\al}-I)^k f(x)|\mathrm{d}s\\
            &=\frac{1}{c_{\lambda}^k}\int_0^{1}s^{-\lambda-1}|(e^{-s}P_s^{\al}-I)^k f(x)|\mathrm{d}s+\frac{1}{c_{\lambda}^k}\int_1^{\infty}s^{-\lambda-1}|
            (e^{-s}P_s^{\al}-I)^k f(x)|\mathrm{d}s\\
            &=:I+II.
        \end{align*}
        When $k=n$, the argument is obvious, thus we will assume $k<n$.
        Choose $0<\epsilon<\frac{1}{2}$ such that
        $\lambda+\epsilon<k$.  Via Proposition \ref{pro3.2} we know that $$Lip_{\beta}(\mu_{\al})\subset Lip_{\lambda+\epsilon}(\mu_{\al})\subset
         Lip_{j-\epsilon}(\mu_{\al}) $$  for $
        j=1,2,\ldots,k-1$. Then  by the binomial theorem, we have
        \begin{equation*}
            (e^{-s}P_s^{\al}-I)^k f(x)=\sum_{j=0}^{k}\binom{k}{j}e^{-js}
            (P_s^{\al}-I)^j (e^{-s}-1)^{k-j}f(x).
        \end{equation*}
        By combining the aforementioned equality with (\ref{eq29}), it follows that
        \begin{align*}
            I&\le \frac{1}{c_{\lambda}^k}\int_0^1 s^{-\lambda-1}\sum_{j=0}^k\binom{k}{j}e^{-js}|(e^{-s}-1)^{k-j}||(P_s^{\al}-I)^{j}f(x)|\mathrm{d}s\\
            &\le \frac{1}{c_{\lambda}^k}\sum_{j=0}^k\int_0^1 s^{-\lambda-1}e^{-js}|(e^{-s}-1)^{k-j}|\|(P_s^{\al}-I)^{j}f\|_{\infty}\mathrm{d}s\\
            &\le \frac{1}{c_{\lambda}^k}\sum_{j=0}^k\int_0^1 s^{-\lambda-1}s^{k-j}\|(P_s^{\al}-I)^{j}f\|_{\infty}\mathrm{d}s\\
            &\le \frac{1}{c_{\lambda}^k}\int_0^1 s^{k-\lambda-1}\mathrm{d}s\|f\|_{\infty}+\sum_{j=1}^{k-1}\binom{k}{j}\frac{A_{j-\epsilon}(f)}{c_{\lambda}^k}\int_{0}^{1}
            s^{-\lambda-1}s^{k-j}s^{j-\epsilon}\mathrm{d}s\\
            &\quad+\frac{A_{\lambda+\epsilon}(f)}{c_{\lambda}^k}
            \int_0^1 s^{\lambda+\epsilon-\lambda-1}\mathrm{d}s\\
            &=\frac{C}{c_{\lambda}^k(k-\lambda)}\|f\|_{\infty}+
            \sum_{j=1}^{k-1}\binom{k}{j}\frac{A_{j-\epsilon}(f)}{c_{\lambda}^k(k-\lambda-\epsilon)}+\frac{A_{\lambda+\epsilon}(f)}{c_{\lambda}^k\epsilon}\\
            &\le C_{\lambda}\|f\|_{Lip_{\beta}(\mu_{\al})}.
        \end{align*}
       We use the following binomial theorem
        \begin{equation}\label{eq37}
            (e^{-s}P_{s}^{\al}-I)^k f(x)=\sum_{j=0}^k\binom{k}{j}(e^{-s}P_s^{\al})^{k-j}(-I)^j
            f(x)
        \end{equation}
        to obtain
        \begin{align*}
            II&\le \frac{1}{c_{\lambda}^k}\int_1^{\infty}s^{-\lambda-1}
            \bigg[\sum_{j=0}^k\binom{k}{j}e^{-(k-j)s}\|P_{(k-j)s}^{\al}f\|_{\infty}\bigg]\mathrm{d}s\\
            &\le \frac{\|f\|_{\infty}}
            {c_{\lambda}^k}\int_1^{\infty}s^{-\lambda-1}(1+e^{-s})^k\mathrm{d}s\le \frac{2^k\|f\|_{\infty}}{\lambda c_{\lambda}^k}\le C_{\lambda}
            \|f\|_{Lip_{\beta}(\mu_{\al})},
        \end{align*}
     which deduces  $\bfd f\in L^{\infty}$.
     Finally, it is necessary to validate the Lipschitz condition. By
     Remark \ref{remk3.2}, we only need to consider the case $0<t<1$. Then, using (\ref{eq37}) and the semigroup property
        \begin{align*}
            P_t^{\al}[(e^{-s}P_s^{\al}-I)^kf(x)]&=P_t^{\al}\bigg(
            \sum_{j=0}^k\binom{k}{j}(-1)^j e^{-(k-j)s}P_{(k-j)s}^{\al}f(x)\bigg)\\
            &=\sum_{j=0}^{\infty}\binom{k}{j}(-1)^j e^{-(k-j)s}P_{t+(k-j)s}^{\al}f(x)\\
            &=\sum_{j=0}^{\infty}\binom{k}{j}(-1)^j e^{-(k-j)s}u(x,t+(k-j)s),
        \end{align*}
      which implies that
        \begin{equation*}
            \frac{\partial^n}{\partial t^n}P_t^{\al}[(e^{-s}P_s^{\al}-I)^kf(x)]=\sum_{j=0}^k
            \binom{k}{j}(-1)^je^{-(k-j)s}u^{(n)}(x,t+(k-j)s).
        \end{equation*}
        Similarly, using again (\ref{eq24}), we divided the estimation of $\frac{\partial^n(P_t^{\al}\bfd f(x))}{\partial t^n}$ into two parts:
        \begin{align*}
            \frac{\partial^n(P_t^{\al}\bfd f(x))}{\partial t^n}&=\frac{1}{c_{\lambda}^k}\frac{\partial^n}{\partial t^n}
            \bigg[\int_0^{\infty}s^{-\lambda-1}P_t^{\al}[(e^{-s}P_s^{\al}-I)^kf(x)]\mathrm{d}s\bigg]\\
            &=\frac{1}{c_{\lambda}^k}\frac{\partial^n}{\partial t^n}
            \bigg[\int_0^{t}s^{-\lambda-1}P_t^{\al}[(e^{-s}P_s^{\al}-I)^kf(x)]\mathrm{d}s\bigg]\\
            &\quad+\frac{1}{c_{\lambda}^k}\frac{\partial^n}{\partial t^n}
            \bigg[\int_t^{\infty}s^{-\lambda-1}P_t^{\al}[(e^{-s}P_s^{\al}-I)^kf(x)]\mathrm{d}s\bigg]\\
            &=\frac{1}{c_{\lambda}^k}\int_0^t s^{-\lambda-1}\sum_{j=0}^{k}\binom{k}{j}(-1)^j e^{-(k-j)s}u^{(n)}(x,t+(k-j)s)\mathrm{d}s\\
            &\quad+\frac{1}{c_{\lambda}^k}\int_t^{\infty} s^{-\lambda-1}\sum_{j=0}^{k}\binom{k}{j}(-1)^j e^{-(k-j)s}u^{(n)}(x,t+(k-j)s)\mathrm{d}s\\
            &=:III+IV.
        \end{align*}
        For $III$, using (\ref{eq28-1}), we can write
        \begin{align*}
            |III|&=\frac{e^t}{c_{\lambda}^k}\bigg|\int_0^t s^{-\lambda-1}
            \sum_{j=0}^k\binom{k}{j}(-1)^je^{-t-(k-j)s}
            u^{(n)}(x,t+(k-j)s)\mathrm{d}s\bigg|\\
            &\le\frac{e^t}{c_{\lambda}^k}\int_0^t s^{-\lambda-1}|\Delta_s^k(e^{-\cdot}u^{(n)}(x,\cdot),t)|\mathrm{d}s.
        \end{align*}
        For the sake of convenience, let us take $f(t)=e^{-t}u^{(n)}(x,t)$, by (\ref{eq15}) and Proposition \ref{pro3.1} we deduce that for any $k>0$
        \begin{equation*}
            \bigg|\frac{\partial^k (u^{(n)}(x,t))}{\partial t^k}\bigg|
            \le Ct^{-(n+k)+\beta}=Ct^{-n+(\beta-k)}.
        \end{equation*}
        Combining this fact with the Leibnitz formula, we get
        \begin{align*}
            \bigg|\frac{\partial^k[e^{-t}u^{(n)}(x,t)]}{\partial t^k}\bigg|&=\bigg|e^{-t}\sum_{j=0}^k\binom{k}{j}(-1)^j
            u^{(n+(k-j))}(x,t)\bigg|\\
            &\le e^{-t}\sum_{j=0}^k\binom{k}{j}|u^{(n+(k-j))}(x,t)|\\
            &\le Ce^{-t}\sum_{j=0}^k\binom{k}{j}t^{-(n+(k-j))+\beta}\\
            &=Ce^{-t}t^{-n+\beta}\sum_{j=0}^k\binom{k}{j}t^{-(k-j)}\\
            &\le Ce^{-t}t^{-n+\beta}2^k t^{-k}=Ce^{-t}t^{-(n+k)+\beta},
        \end{align*}
        with $t\in (0,1)$.
        Then using an idea similar to the one contained in Proposition
        \ref{prop3.5},  we obtain
        \begin{equation*}
            |\Delta_s^k(e^{-\cdot}u^{(n)}(x,\cdot),t)|\le Ct^{-(n+k)+\beta}e^{-t}s^k,
        \end{equation*}
        and consequently,
        \begin{equation*}
            |III|\le \frac{Ct^{-(n+k)+\beta}}{c_{\lambda}^k}\int_0^t
            s^{-\lambda+k-1}\mathrm{d}s=C_{\beta,\lambda}t^{-n+\beta-\lambda}.
        \end{equation*}
        Finally, for $IV$,  using (\ref{eq15}),  we have
        \begin{align*}
            |IV|&\le \frac{1}{c_{\lambda}^k}\int_t^{\infty}s^{-\lambda-1}
            \sum_{j=0}^k\binom{k}{j}e^{-(k-j)s}|u^{(n)}(x,t+(k-j)s)|\mathrm{d}s\\
            &\le \frac{1}{c_{\lambda}^k}\int_t^{\infty}s^{-\lambda-1}
            2^k(t+(k-j)s)^{-n+\beta}\mathrm{d}s\\
            &\le Ct^{-n+\beta}\int_t^{\infty}s^{-\lambda-1}
            \mathrm{d}s=C_{\beta,\lambda}t^{-n+\beta-\lambda}.
        \end{align*}
Putting together all the above estimates, we conclude that
        \begin{equation*}
            \bigg\|\frac{\partial^n}{\partial t^n}(P_t^{\al}\bfd f)\bigg\|_{\infty}
            \le C_{\beta,\lambda}t^{-n+\beta-\lambda}.
        \end{equation*}
Since $\beta-\lambda<n$,  we use Proposition \ref{pro3.1}  to
conclude that $\bfd f\in Lip_{\beta-\lambda}(\mu_{\al})$.
\end{proof}

\subsection*{Acknowledgements}

 {{ J.Z. Huang was supported  by the Fundamental Research Funds for
the Central Universities  (No. 500419772).}  Y. Liu was supported by
the National Natural Science Foundation of China (No. 12271042)  and
Beijing Natural Science Foundation of China (No. 1232023).}

\hspace{-0.65cm}{\bf Address: }

          \flushleft He Wang\\
          School of Mathematics and Statistics\\
          Shandong University of Technology\\
          Zibo 255000, China\\
          E-mail: 1329008486@qq.com

          \flushleft Jizheng Huang\\
          School of Science \\
          Beijing University of Posts and Telecommunications\\
          Beijing 100876, China\\
          E-mail address: hjzheng@163.com

          \flushleft Yu Liu\\
          School of Mathematics and Physics\\
          University of Science and Technology Beijing\\
          Beijing 100083, China\\
          E-mail: liuyu75@pku.org.cn
\end{document}